\documentclass[final]{amsproc}
\usepackage{amsfonts}
\usepackage{amssymb}
\usepackage{latexsym,color}
\usepackage{euscript} 
\usepackage{mathrsfs} 
\DeclareMathAlphabet{\mathpzc}{OT1}{pzc}{m}{it} 
\usepackage{showkeys}

\usepackage{enumerate}

\numberwithin{equation}{section}

\theoremstyle{plain}
\newtheorem{thm}{Theorem}[section]
\newtheorem{pro}[thm]{Proposition}
\newtheorem{lemma}[thm]{Lemma}
\newtheorem{cor}[thm]{Corollary}
\newtheorem*{thm*}{Theorem}
\newtheorem*{pro*}{Proposition}

\theoremstyle{definition}
\newtheorem{dfn}[thm]{Definition}
\newtheorem*{dfn*}{Definition}
\newtheorem{exa}[thm]{{\it Example}}
\newtheorem*{exa*}{{\it Example}}

\theoremstyle{remark}
\newtheorem{rem}[thm]{{\it Remark}}
\newtheorem*{rem*}{{\it Remark}}

\DeclareMathOperator{\supp}{supp}

\newcommand*{\A}{\mathbf{A}}

\newcommand*{\B}{\mathbf{B}}
\newcommand*{\borel}{{\mathfrak B}}
\newcommand*{\bscr}{\mathscr{B}}
\newcommand*{\cbb}{\mathbb C}
\newcommand*{\D}{\mathrm d\hspace{.15ex}}
\newcommand*{\dscr}{\mathscr D}
\newcommand*{\dz}[1]{{\mathscr D}(#1)}
\newcommand*{\dzn}[1]{{\mathscr D}^\infty(#1)}

\newcommand*{\hh}{{\mathcal H}}
\newcommand*{\is}[2]{\langle#1,#2\rangle}
\newcommand*{\jd}[1]{{\mathscr N}(#1)}

\newcommand*{\nbb}{\mathbb N}
\newcommand*{\ob}[1]{{\mathscr R}(#1)}
\newcommand*{\ogr}[1]{\mathbf{B} (#1)}

\newcommand*{\rbb}{\mathbb R}
\newcommand*{\rzut}[1]{\mathbf{P} (#1)}

\newcommand*{\shk}{\mathcal{S}_c({\mathcal H},\kappa)}

\newcommand*{\Le}{\leqslant}

\newcommand*{\Lec}{\preccurlyeq}
\newcommand*{\Ge}{\geqslant}

\newcommand*{\zbb}{\mathbb Z}

\begin{document}

   \title[Multidimensional spectral order]
{Multidimensional spectral order for selfadjoint operators}

   \author[Artur P{\l}aneta]{Artur P{\l}aneta}

   \address{Katedra Zastosowa\'{n} Matematyki, Uniwersytet Rolniczy w Krakowie,
ul. Balicka 253 c, PL-30198 Kra\-k\'ow}

   \email{artur.planeta@urk.edu.pl}

\thanks{The author was supported by the Ministry of Science and Higher Education of the Republic
of Poland.}

\subjclass{}

\keywords{spectral order, joint spectral measure, joint bounded vectors, integral inequalities, separately increasing function}

   \begin{abstract}
The aim of this paper is to extend the notion of the spectral order for finite families of pairwise commuting bounded and unbounded selfadjoint operators in Hilbert space. It is shown that the multidimensional spectral order $\Lec$ is preserved by transformations represented by spectral integrals of separately increasing Borel functions on $\rbb^\kappa$. In particular, the $\kappa$-dimensional spectral order is the restriction of product of $\kappa$ spectral orders for selfadjoint operators.  In the context of positive families of pairwise commuting selfadjoint operator, the relation $\A\Lec\B$ holds if and only if $\A^\alpha\Le \B^\alpha$ for every $\alpha\in\zbb_+^\kappa$.
   \end{abstract}
   \maketitle
 \section{Introduction}

   Let $\hh$ be a complex Hilbert space. Denote by $\mathbf{B}_s(\hh)$ the set of all bounded selfadjoint operators on $\hh$. If $A,B\in\mathbf{\mathbf{B}}_s(\hh)$, then we write $A\Le B$ whenever $\is{Ah}{h}\Le \is{Bh}{h}$ for every $h\in\hh$.

    In \cite{sherman1951} Sherman proved that $\mathcal{A}_s$ the set
of all selfadjoint elements of a $C^*$-algebra $\mathcal{A}$ of bounded linear operators on $\hh$ is a lattice with respect to $\Le$ if and only if $\mathcal{\mathcal{A}}$ is
commutative. In fact, the more noncommutative $\mathcal{A}$ is the less lattice structure $\mathcal{A}_s$ has. The result of
   Kadison \cite{kadison1951} shows that partially ordered set $(\mathbf{\mathbf{B}}_s(\hh),\Le)$ is an anti-lattice. This means that for any $A,B\in \mathbf{\mathbf{B}}_s(\hh)$, the greatest lower bound of the set $\{A,B\}$ exists if and only if $A\Le B$ or $B\Le A$.

  These results motivated Olson to introduce spectral order $\Lec$ for bounded selfadjoint operators \cite{Olson1971}. One of the main result in \cite{Olson1971} says that $(\mathbf{\mathbf{B}}_s(\hh),\Lec)$ is a conditionally complete lattice. It should also be mentioned that in contrast to the classical order $\Le$, the spectral order is not a vector order.  Spectral order was considered in the context of matrix theory and von Neumann algebras (see \cite{Kato1979,Ando1989,arveson1974,Akemann1996}). The spectral order also enabled to solve Dirichlet problem for operator-valued harmonic function in \cite{gogus}.

    The spectral order has natural interpretation in the mathematical description of quantum mechanics (see \cite{Hamhalter2007}). This fact has led to increased interest in automorphisms of the spectral order. For instance, automorphisms of partially ordered set $(\mathcal{E}(H),\Lec)$, where $\mathcal{E}(\hh)$ denote the set of all positive bounded operators in the closed unit ball of $\ogr{\hh}$, were described in \cite{Molnar2007,Molnar2016}. In turn, automorphisms of subsets of positive selfadjoint operators were investigated in \cite{turilova2014,hamhalter-turilova2018}. Spectral order is also an important ingredient of the topos formulation of quantum theory (see for example \cite{Doring,Doring-Dewitt,Wolters2014,Flori2013}).

In the previous paper \cite{Paneta2012} (see also \cite{hamhalter-turilova2018,turilova2014}) the spectral order was investigated in the context of unbounded selfadjoint operators. The aim of this paper is to extend the notion of the spectral order also for finite families of pairwise commuting bounded and unbounded selfadjoint operators. We start with the definition of the spectral order for selfadjoint operators. Let $A$ and $B$ be selfadjoint operators in $\hh$. If $E_A$ and $E_B$ are spectral measures of $A$ and $B$ defined on Borel subsets of $\rbb$, then
    $A\Lec B$ if and only if $E_B((-\infty,x])\Le E_A((-\infty,x])$ for every $x\in\rbb$. As we see, the definition of spectral order depends on the fact that for every selfadjoint operator $A$, there  exists a unique spectral measure $E$ on the $\sigma$-algebra of all Borel subsets of real line satisfying the condition
$$A=\int_{\rbb}xE(\D x).$$ If $\A=(A_1,\ldots,A_\kappa)$ is a $\kappa$-tuple of selfadjoint operators, $\kappa\in\nbb$, then similar representation
\begin{equation}\label{aj}A_j=\int_{\rbb^\kappa}x_jE(\D x),\quad j=1,\ldots,\kappa,\end{equation} holds for some Borel spectral measure $E$ on $\rbb^\kappa$  if and only if $A_1,\ldots,A_\kappa$ pairwise commute.  The set of all $\kappa$-tuples of pairwise commuting selfadjoint operators in $\hh$ will be denoted by $\shk$. By definition, $\A=(A_1,\ldots,A_\kappa)\in\shk$ if spectral measures $E_{A_1},\ldots,E_{A_\kappa}$ commute, i.e.
\begin{equation*} E_{A_k}(\Delta_1)E_{A_j}(\Delta_2)= E_{A_j}(\Delta_2)E_{A_k}(\Delta_1)\end{equation*} for every Borel sets $\Delta_1,\Delta_2\subset\rbb$ and $k,j\in\{1,\ldots,\kappa\}$.
In fact, condition \eqref{aj} determine the unique measure $E=E_\A$, which allow us to consider the multivariable spectral resolution $F_\A(x_1,\ldots,x_\kappa)=E_\A((-\infty,x_1]\times\ldots\times(-\infty,x_\kappa])$ for $x_1,\ldots,x_\kappa\in\rbb$. 
 In particular, if $\A,\B\in\shk$, then they can be ordered by $\A\Lec \B$ if $F_\B(x)\Le F_\A(x)$ for every $x\in\rbb^\kappa$.
This is the first argument for the choice of $\shk$.

The second argument comes from quantum theory. Let $o_1,\ldots,o_\kappa$ and $p_1,\ldots,p_\kappa$ be two systems of commensurable observables represented by $\A,\B\in\shk$, respectively. Consider unit vector $h$ in $\hh$, which is interpreted as the state of the system. Then the pairs $(h,\A)$ and $(h,\B)$ determine random vectors $X=(X_1,\ldots,X_\kappa)$ and $Y=(Y_1,\ldots,Y_\kappa)$, respectively. In particular, the distribution functions of $X$ and $Y$ are given by $F_X(x)=\is{F_\A(x) h}{h}$ and $F_Y(x)=\is{F_\B(x) h}{h}$, $x\in\rbb^\kappa$. Concluding, the relation $\A\Lec\B$ means that the corresponding distribution functions $F_X$ and $F_Y$ are pointwise ordered. For more details on quantum theory, we refer reader to \cite{Bongaarts}.

In Section \ref{swp} we investigate fundamental properties of multidimensional spectral order. The main results state that the multidimensional spectral order is preserved by transformation represented by spectral integrals of separately increasing Borel functions on $\rbb^\kappa$ (see Theorem \ref{nowe} and Theorem \ref{spekwiel}), which extends the known results for spectral order (confer \cite[Proposition 6.2.]{Paneta2012},\cite{Molnar2007}). In particular, we obtain equivalent definition of multidimensional spectral order in Corollary \ref{def-eq}, which is a counterpart of \cite[Theorem 1.]{fujii-kasahara}. As the consequence of Theorem \ref{nowe} we derive some information when the spectral order behaves like a vector order in Corollaries \ref{summno} and \ref{iloczyn}. According to Theorem \ref{spekwiel} it is also possible to extend  the multidimensional spectral order to  $\mathcal{S}(\hh,\kappa)$ the set of all $\kappa$-tuples of selfadjoint operators in $\hh$ as a product order. It means that $\A\Lec\B$ if and only if $A_j\Lec B_j$ for every $j=1,\ldots,\kappa$ whenever $\A,\B\in\mathcal{S}(\hh,\kappa)$. Note that $(\mathcal{S}(\hh,\kappa),\Lec)$ is a conditionally complete lattice since $(\mathcal{S}(\hh,1),\Lec)$ is a conditionally complete lattice by \cite[Corollary 5.4.]{Paneta2012}. In contrast to $(\mathcal{S}(\hh,\kappa),\Lec)$, partially ordered set $(\shk,\Lec)$ is not even a lattice (see Example \ref{lattice}). It should also be pointed out that in the case $\kappa=2$ the multidimensional spectral order on $\mathcal{S}_c(\hh,2)$ can be identified with the spectral order for normal operators (see Example \ref{normalny1} and Proposition \ref{normalny2}). This order was considered by Brenna and Flori in \cite{Brenna2012} in the context of the daseinisation of normal operators.

In Section \ref{dziedzina} we discuss the relation between domains of monomials $\A^\alpha$ and $\B^\alpha$ for $\alpha\in\zbb_+^\kappa$  provided that $\A,\B\in\shk$ and $\A\Lec \B$. The role of operators $\A^\alpha$ and $\B^\alpha$ becomes more important if we consider positive $\A, \B\in\shk$. This question is studied in Section \ref{kolejny}. One of the main result of this Section states that $\A\Lec \B$ if and only if $\A^\alpha\Le \B^\alpha$ for every $\alpha\in\zbb_+^\kappa$. This generalize \cite[Theorem 3.]{Olson1971}, which says that $A\Lec B$ if and only if $A^n\Le B^n$ for every $n=1,2,\ldots$, where $A$ and $B$ are bounded positive selfadjoint operators on $\hh$.

   \section{Prerequisites}
Denote by $\zbb_+$, $\nbb$, $\mathbb{Q}$, $\rbb_+$, $\rbb$, and $\cbb$ the
sets of nonnegative integers, positive integers, rational numbers,
nonnegative real numbers, real numbers, and complex numbers,
respectively. Set $\overline{\rbb} = \rbb \cup
\{-\infty\} \cup \{\infty\}$. As usual, $\chi_\sigma$
stands for the characteristic function of a set
$\sigma$ (it is clear from the context which set is
the domain of definition of $\chi_\sigma$).

In what follows, $\hh$ denotes a complex Hilbert
space. By an {\em operator} in $\hh$ we understand a
linear mapping $A\colon \hh \supseteq \dz A
\rightarrow \hh$ defined on a linear subspace $\dz A$
of $\hh$, called the {\em domain} of $A$. Let $A$ be
an operator in $\hh$. We write $\jd A$, $\ob A$, $A^*$
and $\bar A$ for the kernel, the range, the adjoint
and the closure of $A$, respectively (in case they
exist). Set $\dzn{A}=\bigcap_{n=1}^{\infty}\dz{A^{n}}$
and
   \begin{equation*}
\bscr_{a}(A)= \bigcup_{c \in (0,\infty)} \;
\bigcap_{n=0}^\infty \Big\{ h\in \dzn{A}\colon \Vert
A^{n}h\Vert\Le ca^{n}\Big\}, \quad a \in \rbb_+.
   \end{equation*}
Call a member of $\bscr(A):=\bigcup_{a\in \rbb_+}
\bscr_{a}(A)$ a {\em bounded vector} of $A$ (cf.\
\cite{far}). We say that an operator $A$ in $\hh$ is
{\em positive} if $\is{Ah}h \Ge 0$ for every $h \in
\dz A$. A densely defined operator $A$ is said to be
{\em selfadjoint} if $A=A^*$. If $A$ is a positive
selfadjoint operator in $\hh$, then its square root is
denoted by $A^{1/2}$.

Denote by $\ogr \hh$ the $C^*$-algebra of all bounded
operators $A$ in $\hh$ such that $\dz A=\hh$. We write
$I=I_\hh$ for the identity operator on $\hh$. 
 It is well-known that the set
$\rzut{\hh}$ of all orthogonal projections on $\hh$
equipped with the above-defined partial order
``$\Le$'' is a complete lattice. If
$\mathcal{A}\subseteq\rzut{\hh}$, then $\bigvee
\mathcal{A}$ (respectively, $\bigwedge \mathcal{A}$)
stands for the supremum (respectively, the infimum) of
$\mathcal{A}$. Denote by $P_M$ the orthogonal
projection of $\hh$ onto a closed linear subspace $M$
of $\hh$. If $\{M_j\}_{j\in J}$ is a family of subsets
of $\hh$, then $\bigvee_{j\in J} M_j$ stands for the
smallest closed linear subspace of $\hh$ containing
$\bigcup_{j\in J}M_j$. Note that if each $M_j$ is a
closed linear subspaces of $\hh$, then
   \begin{equation}\label{sup1}
\bigvee \big\{P_{M_j} \colon j\in J\big\} =
P_{\,\bigvee_{j\in J} M_j} \quad \text{and} \quad
\bigwedge\big\{P_{M_j} \colon j\in J\big\} =
P_{\bigcap_{j\in J} M_j}.
   \end{equation}
We shall use the symbol ``$\textsc{sot-}\lim$'' to
represent the limit in the strong operator topology on
$\ogr{\hh}$.

Let $E\colon \borel(X) \to
\ogr{\hh}$ be a spectral measure defined on the $\sigma$-algebra $\borel(X)$
of all Borel subsets of a topological space $X$.
 A set $Y\subseteq X$
 is called a \textit{support} of $E$, if $Y$ is the least (with
 respect to set inclusion) closed subset of $X$ such
 that $E(X\backslash Y)=0$. The support of $E$ is
 denoted by $\supp E$. It is known that every
 spectral measure $E$ on $\borel(X)$, where $X$ is a
 separable complete metric space, has closed support
 (cf. \cite[page 129]{Bir-Sol}). We say
that $E$ is \textit{regular} if
 \begin{equation*}
E(\sigma)=\bigvee\{E(\tau)\colon \tau\subseteq\sigma,
\tau \textmd{ is compact in } X\} \textmd{ for every }
\sigma\in\borel(X).
\end{equation*}
For the sake of completeness, we prove the following Proposition (see also \cite[Exercise 6.8.]{MacCluer2009}).
\begin{pro}\label{regular1}
 If $X$ is a separable complete metric space and $E$
is a spectral measure on $\borel(X)$, then $E$ is
regular.
\end{pro}
\begin{proof}
 Let $\sigma\in\borel(X)$. It is evident that $$\bigvee\{E(\tau)\colon
\tau\subseteq\sigma, \tau \textmd{ is compact}\}\Le
E(\sigma).$$ To establish the opposite inequality, take
$h\in\hh$ such that $h\perp E(\tau)\hh$ for every
compact set $\tau\subseteq\sigma$. Applying
\cite[Subsection 1.3.22]{Bir-Sol} we get that
$\mu_h(\cdot):=\is{E(\cdot)h}{h}$ is an inner regular Borel
measure (cf. \cite{Rudin}). Choose $g\in\hh$. Then
\begin{align*}
|\is{h}{E(\sigma)g}|&=|\is{E(\sigma)h}{g}|\Le |\is{E(\sigma)h}{h}|\cdot\|g\|\\
&=\sup\big\{\is{E(\tau)h}{h}\colon \tau\subseteq\sigma,
\tau \textmd{ is compact}\big\}\cdot\|g\|=0.
\end{align*}
Thus $h\perp E(\sigma)\hh$, since $g$ was chosen arbitrarily. This completes the proof.
\end{proof}

   Let us recall that
$\supp E_A=\sigma(A)$ for every selfadjoint operator
$A$ in $\hh$, where $\sigma(A)$ denotes the spectrum
of $A$ and $E_A$ is the spectral measure of $A$. In particular, a selfadjoint operator $A$ is
positive if and only if $\supp E_A\subseteq[0,\infty)$
(cf. \cite{Bir-Sol}).
If $\varphi\colon\rbb\to\rbb$ is a Borel function,
then we set
   \begin{align*}
\varphi(A) = \int_\rbb \varphi(x) E_A(\D x).
   \end{align*}
The operator $\varphi(A)$ is selfadjoint. Moreover, if
$\varphi \Ge 0$ a.e.\ $[E]$, then $\varphi(A)$ is
positive. For more information concerning Stone-von
Neumann operator calculus $\varphi \mapsto
\varphi(A)$, we refer the reader to \cite{Bir-Sol,Schmudgen}.

Given a positive selfadjoint operator $A$ in $\hh$ and
$s \in \rbb_+$, we define $A^s=\psi_s(A)$, where
$\psi_s(x) = |x|^s \chi_{[0,\infty)}(x)$ for $x \in
\rbb$ (with the convention that $0^0=1$). This
definition agrees with the usual one for nonnegative
integer exponents. If $A_1$ and $A_2$ are positive
selfadjoint operators in $\hh$ such that
$\dz{A_2^{1/2}} \subseteq \dz{A_1^{1/2}}$ and
$\|A_1^{1/2} h\| \Le \|A_2^{1/2}h\|$ for every $h \in
\dz{A_2^{1/2}}$, then we write $A_1 \Le A_2$ (cf.\
\cite{Schmudgen}). The last definition is easily seen to be
consistent with that for bounded operators.

Finally, let $(X,\mathcal{A},\mu)$ be a measure space and $\varphi\colon X\to\rbb$ be an  $\mathcal{A}$-measurable function. We denote by  $M_{\varphi}$  the multiplication operator by $\varphi$ in $L^2(X,\mu)$, which is defined as follows:
\begin{align*}
\dz{M_{\varphi}}:=\{h\in L^2(X,\mu)\colon \int_X\vert \varphi h\vert^2 \, \textrm{d}\mu<\infty\},\\
M_{\varphi}h:=\varphi h, \quad h\in \dz{M_{\varphi}}.
\end{align*}
It is known that $M_{\varphi}$ is selfadjoint and
\begin{equation}\label{emphi}E_{M_\varphi}(\sigma)h=\chi_{\varphi^{-1}(\sigma)}h
\end{equation}
for every $\sigma\in\borel(\rbb)$ and $h\in L^2(X,\mu)$ (see \cite[Example 5.3.]{Schmudgen}).
   \section{Separately increasing functions}
Here we collect some facts about separately increasing functions we need in this
paper.

Fix $\kappa\in\nbb$. In what follows, $\rbb^{\kappa}$
is regarded as the $\kappa$-dimensional Euclidean
space. Denote by $\{e_j\}_{j=1}^{\kappa}$ the standard
orthonormal basis of $\rbb^{\kappa}$, i.e. $$e_j =(0,\ldots,0,\underbrace{1}\limits_j,0,\ldots,0),\quad j=1,\ldots,\kappa.$$

Throughout this paper, we adhere to the usual
convention that if $a \in \overline{\rbb}^{\kappa}$,
then the jth coordinate of $a$ is denoted by $a_j$;
thus $a = (a_1,\ldots, a_{\kappa})$. Let $a
\in\overline{\rbb}^{\kappa}$ and
$b\in\overline{\rbb}^{\kappa}$. We write $a\Le b$
(resp., $a<b$) if $a_j\Le b_j$ (resp., $a_j< b_j$) for
all $j=1,\ldots,\kappa$. Set
   \begin{align*}
(a,b] & = \{x\in\overline{\rbb}^{\kappa}\colon a< x\Le
b\},
   \\
[a,b] & = \{x\in\overline{\rbb}^{\kappa}\colon a\Le
x\Le b\},
   \\
(-\infty,x]&=\{y\in\rbb^{\kappa}\colon y\Le
x\}=(-\infty,x_1]\times\ldots\times(-\infty,x_{\kappa}],
\quad x \in\rbb^{\kappa},
   \end{align*}
Similarly, we define $(-\infty,x)$, $[x, \infty)$ and
$(x,\infty)$. Given $\iota\in\{1,\ldots,\kappa\}$, $a
\in\rbb^{\kappa}$ and $b\in\rbb^{\kappa}$, we write
$a\Le_\iota b$ if $a_j\Le b_j$ for $j=1,\ldots,\iota$
and $a_j= b_j$ for $j>\iota$. Note that
$(\overline{\rbb}^{\kappa}, \Le)$ and
$(\rbb^{\kappa},\Le_\iota)$ are partially ordered
sets.

Let $(X,\preceq)$ be a partially ordered set. A set
$S\subseteq X$ is called a {\em lower set} in $X$ if
   \begin{equation*}
\{y\in X\colon y \preceq x\} \subseteq S, \quad x\in
S.
   \end{equation*}
The collection of all lower sets in $X$ is denoted by
$L(X,\preceq)$. Note that $\bigcap \Gamma \in
L(X,\preceq)$ for any collection $\Gamma\subseteq
L(X,\preceq)$. Hence, if $\varOmega \subseteq X$, then
the following set
   \begin{equation*}
\downarrow \varOmega:=\bigcap\big\{S\in
L(X,\preceq)\colon \varOmega\subseteq S\big\},
   \end{equation*}
is the smallest lower set in $X$ containing
$\varOmega$. The following properties of lower sets
are easily seen to be true.
   \begin{pro}\label{cwiartka}
If $(X,\preceq)$ is a partially ordered set and
$\varOmega, \varOmega_1, \varOmega_2\subseteq X$, then
   \begin{enumerate}
   \item[{\em(i)}] $\downarrow\varOmega = \bigcup_{x\in \varOmega}\{y\in X\colon
y\preceq x\}$,
   \item[{\em(ii)}] if $\varOmega_1 \subseteq \varOmega_2$, then
$\downarrow \varOmega_1 \subseteq \; \downarrow
\varOmega_2$.
   \end{enumerate}
   \end{pro}
In the case of the partially ordered set
$(\rbb^\kappa,\Le_\iota)$, we write $\downarrow_\iota
S$ in place of $\downarrow S$. The following fact is
of some importance in this paper.
   \begin{pro}\label{cc}
If $\varOmega\subseteq\rbb^{\kappa}$ is compact, then
$\downarrow_\iota \varOmega$ is closed for every
$\iota\in\{1,\ldots,\kappa\}$.
   \end{pro}
   \begin{proof}
Fix $x\in \overline{\downarrow_\iota\varOmega}$. Then
there exists a sequence $\{x_n\}_{n=1}^{\infty}
\subseteq \; \downarrow_\iota\varOmega$ such that
$x_n\to x$ as $n\to\infty$. By Proposition
\ref{cwiartka}(i), we can find a sequence
$\{y_n\}_{n=1}^{\infty}\subseteq\varOmega$ such that
$x_n\Le_{\iota} y_n$ for all $n\in\nbb$. By the
compactness of $\varOmega$, there exists a subsequence
$\{y_{n_k}\}_{k=1}^{\infty}$ of
$\{y_n\}_{n=1}^{\infty}$ which converges to some
$y\in\varOmega$. In particular, we have
   \begin{align*}
x=\displaystyle\lim_{k\to\infty} x_{n_k}\Le_{\iota}
\displaystyle\lim_{k\to\infty}y_{n_k}=y.
   \end{align*}
Applying Proposition \ref{cwiartka}(i) again, we see
that $x\in \downarrow_\iota\varOmega$.
   \end{proof}
   \begin{rem}
First, we note that if $\varOmega \neq \varnothing$,
then the set $\downarrow_\iota \varOmega$ is always
unbounded (see Proposition \ref{cwiartka}(i)). It is
also worth pointing out that the assumption about
compactness of $\varOmega$ in the Proposition \ref{cc}
is essential. Indeed, if $\kappa=2$ and
$\varOmega=\{(-\frac{1}{n},n)\colon n\in \nbb\}$, then
by Proposition \ref{cwiartka}(i),
$\downarrow_2\varOmega=(-\infty,0)\times \rbb$. This
means that $\downarrow_2\varOmega$ is not closed,
though $\varOmega$ is.
   \end{rem}
The following terminology will be frequently used in
this paper.
   \begin{dfn} Let $\kappa_1, \kappa_2\in\nbb$, $\iota \in
\{1, \ldots, \kappa_1\}$ and
$\varOmega\subseteq\rbb^{\kappa_1}$. We say that a
function $\varphi\colon \varOmega\rightarrow\rbb^{\kappa_2}$ is {\em
$\iota$-increasing} if
   \begin{equation*}
\text{$x\Le_\iota y \Rightarrow \varphi(x)\Le
\varphi(y)$ for all $x,y\in\varOmega$.}
   \end{equation*}
The function $\varphi$ is called {\em increasing} if
it is $\kappa_1$-increasing.
   \end{dfn}
As shown below, lower sets can be used to characterize
increasing functions.
   \begin{pro}\label{oddzielros}
Let $\kappa_1,\kappa_2\in\nbb$ and $\iota \in \{1,
\ldots, \kappa_1\}$. Then a function $\varphi\colon
\rbb^{\kappa_1}\rightarrow\rbb^{\kappa_2}$ is
$\iota$-increasing if and only if
   \begin{align*}
\text{$\{\varphi\Le y\} \in
L(\rbb^{\kappa_1},\Le_\iota)$ for all
$y\in\rbb^{\kappa_2}$ $($equivalently$\colon$ for all
$y\in \varphi(\rbb^{\kappa_1})$$)$,}
   \end{align*}
where $\{\varphi\Le y\}:=\{x\in\rbb^{\kappa_1}\colon
\varphi(x)\Le y\}$.
   \end{pro}
   \begin{proof} It is enough to prove the sufficiency.
Take $x_1,x_2\in\rbb^{\kappa_1}$ such that
$x_1\Le_\iota x_2$. Set $\varOmega=\{\varphi\Le
\varphi(x_2)\}$. Since $\varOmega\in
L(\rbb^{\kappa_1},\Le_\iota)$ and $x_2\in\varOmega$,
we see that $x_1\in\varOmega$. Thus $\varphi(x_1)\Le
\varphi(x_2)$.
   \end{proof}
   Now we prove that the distance function to a
$\iota$-lower set is $\iota$-increasing.
   \begin{pro}\label{oepsilon}
Let $\kappa \in \nbb$ and
$\iota\in\{1,\ldots,\kappa\}$. Suppose \mbox{$\Vert
\cdot \Vert$} is a norm on $\rbb^{\kappa}$ and
$\varOmega \in L(\rbb^\kappa,\Le_\iota)$. If
$\varOmega \neq \varnothing$, then
   \begin{enumerate}
\item[{\em(i)}] the function $d_\varOmega\colon
\rbb^{\kappa}\to \rbb$ defined by
   \begin{align*}
d_{\varOmega}(x)=\inf\{\Vert x-y\Vert\colon
y\in\varOmega \}, \quad x\in\rbb^\kappa,
   \end{align*}
is $\iota$-increasing,
\item[{\em(ii)}]  $\varOmega^{(\varepsilon)}:=
\{d_{\varOmega} \Le \varepsilon\}\in
L(\rbb^\kappa,\Le_\iota)$ for every $\varepsilon \in
(0,\infty)$.
   \end{enumerate}
   \end{pro}
   \begin{proof}(i) Let $x,y\in\rbb^\kappa$ be
such that $x\Le_\iota y$. Fix $\varepsilon \in
(0,\infty)$ and take $y_{\varepsilon}\in\varOmega$
such that $\Vert
y-y_{\varepsilon}\Vert<d_{\varOmega}(y)+\varepsilon$.
Then $x_{\varepsilon}:=y_{\varepsilon}+x-y\Le_\iota
y_{\varepsilon}$, which implies that
$x_{\varepsilon}\in\varOmega$ (because $\varOmega\in
L(\rbb^\kappa,\Le_\iota)$). Since $\Vert
x-x_{\varepsilon}\Vert=\Vert
y-y_{\varepsilon}\Vert<d_{\varOmega}(y)+\varepsilon$,
we deduce that $d_{\varOmega}(x) <
d_{\varOmega}(y)+\varepsilon$. This yields
$d_{\varOmega}(x)\Le d_{\varOmega}(y)$.

(ii) This follows from (i) and Proposition
\ref{oddzielros}.
   \end{proof}
We conclude this section with an example showing that
increasing function $f\colon\rbb^\kappa\to\rbb$ may
not be Borel (this is possible only for $\kappa\Ge 2$,
see \cite[Lemma 6.1.]{Paneta2012}).
   \begin{exa}
Let $V$ be a subset of $\rbb$ which is not Borel
measurable (cf.\ \cite[Corollary, p.\ 53]{Rudin}). Set
$S=\{(x,y)\in\rbb^2\colon x+y=0 \}$,
$S_1=\{(x,y)\in\rbb^2\colon x+y>0\}$ and $S_2=S\cap
(V\times \rbb)$. It is easily seen that the function
$f:=\chi_{S_1\cup S_2}$ is increasing. If $f$ were a
Borel function, then the set $f^{-1}(\{1\})=S_1\cup
S_2$ would be a Borel subset of $\rbb^2$. This and the
fact that $S_1$ is an open set in $\rbb^2$ which is
disjoint from $S_2$ would imply that
$S_2\in\borel(\rbb^2)$. Since the function
$\varphi\colon \rbb\ni x\rightarrow (x,-x)\in\rbb^2$
is continuous and $V =\varphi^{-1}(S_2)$, this would
give a contradiction.
   \end{exa}

\section{Some integral inequalities on $\rbb^\kappa$}
In this section we provide integral inequalities for separately incresing function. These inequalities are related to the multidimensional spectral order. The main result of this section is Theorem \ref{3.1}, which is a generalization of \cite[Theorem 4.1.]{Paneta2012}. In the proof of this theorem, we will need the following lemma which characterizes
regularity of lower sets.
\begin{lemma}\label{regcwiar}
Let $\mu$ be a finite Borel measure on $\rbb^\kappa$
and let $\iota\in\{1,\ldots,\kappa\}$. If $\varOmega \in
L(\rbb^\kappa,\leq_\iota)\cap\borel(\rbb^\kappa)$, then
\begin{equation*}
\mu(\varOmega)=\sup\{\mu(D)\colon D\subseteq\varOmega,
D=\overline{D}\in L(\rbb^\kappa,\leq_\iota)\}.
\end{equation*}
\end{lemma}
\begin{proof}
Let $\varOmega \in
L(\rbb^\kappa,\leq_\iota)\cap\borel(\rbb^\kappa)$.
Applying \cite[Theorem 2.18.]{Rudin}, Proposition
\ref{cwiartka}, and Proposition \ref{cc}, we get that
\begin{align*}
 \mu(\varOmega)&=\sup\{\mu(K)\colon K\subseteq\varOmega,\, K\textmd{ is compact} \}\\
&\Le\sup\{\mu(\downarrow_\iota K)\colon K\subseteq\varOmega,\, K\textmd{ is compact}\}\\
&\Le\sup\{\mu(D)\colon
D\subseteq\varOmega,\, D=\overline{D}\in L(\rbb^\kappa,\leq_\iota) \}\Le\mu(\varOmega).
\end{align*}
This completes the proof.
\end{proof}
\begin{thm}\label{3.1}
   Let $\mu_1$ and $\mu_2$ be a finite Borel measure
   on $\rbb^\kappa$. We consider the following
   conditions\/{\em :}
   \begin{enumerate}[{\em(i)}]
   \item $\mu_2(\varOmega) \Le \mu_1(\varOmega)$ for every $\varOmega \in L(\rbb^\kappa,\leq_\iota)\cap\borel(\rbb^\kappa)$,
\item $\int_{\rbb^\kappa} f \D \mu_{1}\Le
\int_{\rbb^\kappa} f \D \mu_{2}$ for every $\iota$-increasing Borel
function $f\colon\rbb^\kappa\to[0,\infty)$,
   \item $\int_{\rbb^\kappa} f \D \mu_{1}\Le
\int_{\rbb^\kappa} f \D \mu_{2}$ for every $\iota$-increasing Borel
function $f\colon\rbb^\kappa\to\rbb$ such
that $\int_{\rbb^\kappa} |f| \D \mu_{2} < \infty$
$($the integral $\int_{\rbb^\kappa} f \D \mu_{1}$ may
be $-\infty$$)$,
 \item $\int_{\rbb^\kappa} f \D \mu_{1}\Le
\int_{\rbb^\kappa} f \D \mu_{2}$ for every $\iota$-increasing bounded
continuous function\\ $f\colon\rbb^\kappa\to\rbb$,
 \item $\int_{\rbb^\kappa} f \D \mu_{1}\Le
\int_{\rbb^\kappa} f \D \mu_{2}$ for every nonnegative $\iota$-increasing bounded
continuous function $f\colon\rbb^\kappa\to[0,\infty)$.
   \end{enumerate}
Then
\begin{enumerate}[\em(a)]
\item condition {\em (i)} $($respectively
{\em (ii)}, {\em (iii)}$)$ implies that
$\mu_2(\rbb^\kappa) \Le \mu_{1}(\rbb^\kappa)$
$($respectively $\mu_1(\rbb^\kappa) \Le
\mu_{2}(\rbb^\kappa)$, $\mu_1(\rbb^\kappa) =
\mu_{2}(\rbb^\kappa)$$)$.
\item Inequality
$\mu_{1}(\rbb^\kappa) \Le \mu_{2}(\rbb^\kappa)$ and
{\em (i)} imply that $\mu_1(\rbb^\kappa) =
\mu_{2}(\rbb^\kappa)$.
\item If $\mu_1(\rbb^\kappa) =
\mu_2(\rbb^\kappa)$, then all conditions {\em (i)-(v)}
are equivalent.
\end{enumerate}
   \end{thm}
\begin{proof} We only prove (c), since (a) and (b) are obvious. Assume that  $\mu_1(\rbb^\kappa) = \mu_{2}(\rbb^\kappa)$.

 (i)$\Rightarrow$(iii) Let $f$ be a real $\iota$-increasing Borel
function
 such that $\int_{\rbb^\kappa} |f| \D \mu_{2} <
 \infty$. Define
 $\nu_j(\sigma):=\mu_j(f^{-1}(\sigma))$ for
 $\sigma\in\borel(\rbb)$ and $j=1,2$. It is evident
 that $\nu_1$ and $\nu_2$ are finite
 Borel measure on $\rbb$ and
 \begin{align}\label{Rn12}
\nu_1(\rbb)=\mu_1(\rbb^\kappa)=\mu_2(\rbb^\kappa)=\nu_2(\rbb).
\end{align}
 Moreover,
\begin{align}\label{ni1ni2}
\nu_2
((-\infty,x])=\mu_2(f^{-1}(-\infty,x]))\Le\mu_1(f^{-1}(-\infty,x]))=\nu_1
((-\infty,x])
\end{align}
for every $x\in\rbb$ since $f^{-1}((-\infty,x])$ is a lower set in $(\rbb^\kappa,\leqslant_\iota)$ by Proposition
 \ref{oddzielros}.
 Applying the measure transport theorem \cite[p. 163, Theorem C.]{Halmos1950}, we obtain
\begin{align*}
 \int_{\rbb} |x| \D \nu_{2}(x)=\int_{\rbb^\kappa} |f|
\D \mu_{2}< \infty.
\end{align*}
Thus, by the measure transport theorem, \eqref{Rn12}, \eqref{ni1ni2}, and \cite[Theorem 4.1.]{Paneta2012},
we have
 \begin{align*}
\int_{\rbb^\kappa} f \D \mu_{1}=\int_{\rbb} x \D
\nu_{1}(x)\Le\int_{\rbb} x \D \nu_{2}(x)
=\int_{\rbb^\kappa} f \D \mu_{2}.
 \end{align*}

(iii)$\Rightarrow$(ii), (iii)$\Rightarrow$(iv), and
(iv)$\Rightarrow$(v) are obvious.

(ii)$\Rightarrow$(i) Let $\varOmega \in
\borel(\rbb^\kappa)\cap L(\rbb^\kappa,\leq_\iota)$.
Then
$f=\chi_{\rbb^\kappa\backslash\varOmega}$ is a $\iota$-increasing Borel function. By assumption we obtain
  \begin{align*}
 \mu_1(\rbb^\kappa\backslash\varOmega)=\int_{\rbb^\kappa} \chi_{\rbb^\kappa\backslash\varOmega} \D \mu_1
 \overset{\textrm{(ii)}}\Le
 \int_{\rbb^\kappa} \chi_{\rbb^\kappa\backslash\varOmega} \D \mu_2
 =\mu_2(\rbb^\kappa\backslash\varOmega).
   \end{align*}
Hence $\mu_2(\varOmega) \Le \mu_1(\varOmega)$, since
$\mu_1(\rbb^\kappa)=\mu_2(\rbb^\kappa)$.

(v)$\Rightarrow$(i) According to Lemma \ref{regcwiar}
it is sufficient to show that
\begin{equation*}
 \mu_2(D)\Le \mu_1(D),
\end{equation*}
 for every closed set $D\in
L(\rbb^\kappa,\leq_\iota)$.

 Fix a closed set $D\in L(\rbb^\kappa,\leq_\iota)$. For every $n\in\nbb$, we
define function
$f_n\colon\rbb^\kappa\to\rbb$ by
$f_n(x):=\min\{1,nd_D(x)\}$, $x\in\rbb^\kappa$.
It is easily seen that $0\Le f_n(x)\Le 1$, $f_n(x)\Le
f_{n+1}(x)$ and
$\lim_{n\to\infty}f_n(x)=\chi_{\rbb^\kappa\backslash
D}(x)$ for every $x\in\rbb^\kappa$. Note that $f_n$'s
are continuous and $\iota$-increasing, since $d_D$
is $\iota$-increasing by Proposition \ref{oepsilon}. It follows from Lebesgue's monotone convergence theorem and assumption that
\begin{align*}
\mu_1(\rbb^\kappa\backslash D)= \lim_{n\to\infty}
\int_{\rbb^\kappa} f_n \D
\mu_{1}\overset{\textrm{(v)}}{\Le}
\lim_{n\to\infty}\int_{\rbb^\kappa} f_n \D \mu_{2} =
\mu_2(\rbb^\kappa\backslash D).
\end{align*}
This and equality $\mu_1(\rbb^\kappa)=\mu_2(\rbb^\kappa)$ imply that
$\mu_2(D)\Le\mu_1(D)$.
\end{proof}
\begin{rem}
   Let $\mu_1$ and $\mu_2$ be finite nonnegative
Borel measures on $\rbb^\kappa$ such that
$\mu_1(\rbb^\kappa)=\mu_2(\rbb^\kappa)$. In this
settings the condition
\begin{equation*}
\mu_2((-\infty,x]) \Le \mu_1((-\infty,x]),\:
x\in\rbb^\kappa,
\end{equation*}
does not have to imply conditions (i)-(v) in
Theorem \ref{3.1}. Indeed, take two Borel measures
$\mu_1$ and $\mu_2$ on $\rbb^2$ defined by the
following formulas
\begin{equation*}
\mu_1:=\delta_{(0,0)}+\delta_{(1,1)} \textmd{ and }
\mu_2:=\delta_{(0,1)}+\delta_{(1,0)},
\end{equation*}
 where $\delta_{(0,0)}$, $\delta_{(0,1)}$,
 $\delta_{(1,0)}$ and $\delta_{(1,1)}$ are Dirac
 measures on $\borel(\rbb^2)$. It is evident, that
 $\mu_1(\rbb^2)=\mu_2(\rbb^2)=2$ and
 $\mu_2((-\infty,x]) \Le \mu_1((-\infty,x])$ for every
 $x\in\rbb^2$. At the same time $\mu_1(S) < \mu_2(S)$,
 where $S=\{(x_1,x_2)\in\rbb^2\colon x_1+x_2\Le 1\}$.
 This means that measures $\mu_1$ and $\mu_2$ do not
 satisfy the condition (i). Hence all of conditions
 (i)-(v) in Theorem \ref{3.1} do not hold.
\end{rem}
   \section{Multidimensional spectral order}\label{swp}
In this section we generalize the notion of spectral
order to the case of finite families of pairwise commuting selfadjoint
operators. First we give necessary background on
joint spectral measures and multivariable resolution of the identity.

 Let $\B_{cs}(\hh,\kappa):=\shk\cap\ogr{\hh}^\kappa$. If $\A=(A_1,\ldots, A_{\kappa})\in\shk$,
then there exists the
unique spectral measure $E_{\A}\colon
\borel(\rbb^{\kappa}) \to \ogr{\hh}$, called the {\em
joint spectral measure} of $\A$, such that
   \begin{align}\label{wspolnamiara}
A_j=\int_{\rbb^{\kappa}} x_j  E_{\A}(\D x), \quad j=1,
\ldots, \kappa.
   \end{align}
In fact, $E_{\A}$ coincides with the product of the
spectral measures of the operators $A_1, \ldots,
A_{\kappa}$ (cf.\ \cite{Bir-Sol}). This relationship can be made more explicit by the following two equivalent equations
\begin{equation}\label{eapij}
E_{A_j}(\sigma)=E_{\A}(\pi_j^{-1}(\sigma)), \quad
\sigma\in\borel(\rbb),\, j=1,\ldots,\kappa,
\end{equation}
where
$\pi_j\colon\rbb^{\kappa}\ni(x_1,\ldots,x_{\kappa})\rightarrow
x_j\in\rbb$, and
\begin{equation}\label{miaraEa}
E_{\textbf{A}}(\sigma_1\times\ldots\times\sigma_{\kappa})=E_{A_1}(\sigma_1)\ldots
E_{A_{\kappa}}(\sigma_{\kappa}),\quad \sigma_1,\ldots,\sigma_{\kappa}\in\borel(\rbb).
\end{equation}

    The {\em joint spectral
distribution} $F_{\A}$ of $\A$ is defined by
   \begin{align}\label{fa}
F_{\A}(x) = E_{\A}((-\infty,x]), \quad
x\in\rbb^{\kappa}.
   \end{align}
As in the one-variable case (cf.\ \cite{Bir-Sol}), one can
build the theory of multivariable resolutions of the identity independently of spectral measures.  In Section \ref{app1}, for
the reader's convenience, we recall the definition of
an abstract multivariable resolution of the identity and we
outline an idea how to construct
spectral measures on $\borel(\rbb^\kappa)$ by using multivariable resolution of the identity.

Let $\A\in\shk$.
Given a Borel function $\varphi\colon\rbb^\kappa\to\rbb$ we define
\begin{equation*}
\varphi(\A)=\int_{\rbb^{\kappa}}\varphi(x)E_{\A}(\D x).
\end{equation*}
 The operator $\varphi(\bf{A})$ is selfadjoint. Moreover, if $\varphi\Ge 0$, then $\varphi(\bf{A})$ is positive.

After these preliminary remarks, we can state the definition of spectral order in multidimensional case.
\begin{dfn}
 Let $\A,\B\in\shk$. We write $\A\Lec \B$, if $F_{\B}(x)\Le F_{\A}(x)$
 for every $x\in\rbb^{\kappa}$.
\end{dfn}
It is easily seen that the relation ''$\Lec$'' is a partial order in $\shk$, since there is a one-to-one correspondence between the set of all resolutions of the identity on
$\rbb^\kappa$ and the set $\shk$. This order will be called \textit{multidimensional spectral order}. If $\kappa=1$, then we obtain the definition of spectral order for selfadjoint operators (cf. \cite{Olson1971},\cite{Paneta2012}).
Now we are going to investigate the fundamental properties of multidimensional spectral order ''$\Lec$''.  First we are going to answer the question which Borel functions $\varphi\colon\rbb^\kappa\to\rbb$ preserve the multidimensional spectral order, i.e. the map $\Phi\colon(\shk,\Lec)\ni\A\to \varphi(\A)\in(\mathcal{S}_c(\hh,1),\Lec)$ is a morphism in the category of partially ordered set. We begin by showing that a Borel function $\varphi\colon\rbb^\kappa\to\rbb$ preserving the multidimensional spectral order has to be increasing.
\begin{pro}
 Let $\kappa\in\nbb$ and let $\hh$ be a Hilbert space
such that $\dim \hh\geqslant 1$. If $\varphi\colon
\rbb^{\kappa}\rightarrow \rbb$ is a Borel function
satisfying the following condition
\begin{equation}\label{niez}
\A\Lec \B\implies \varphi(\A)\Lec \varphi(\B)
\end{equation}
for every $\A,\B\in\shk$, then $\varphi$ is an increasing function.
\end{pro}
\begin{proof}
Fix $a=(a_1,\ldots,a_{\kappa})\in\rbb^{\kappa}$ and
$b=(b_1,\ldots,b_{\kappa})\in\rbb^{\kappa}$ such that
$a\Le b$. Define $A_i=a_iI$ and $B_i=b_iI$ for
$i=1,\ldots,\kappa$. It is easily seen that
$E_{\A}(\sigma)=\delta_a(\sigma)I$ and
$E_{\B}(\sigma)=\delta_b(\sigma)I$ for every
$\sigma\in\borel(\rbb^\kappa)$, where $\delta_c$ is a
Dirac measure defined on $\borel(\rbb^\kappa)$ at the
point $c\in\rbb^\kappa$. Hence $\A\Lec\B$ because
$a\Le b$. Then, by the condition \eqref{niez},
$\varphi(a)I=\varphi(\A)\Lec \varphi(\B)=\varphi(b)I$. Thus $\varphi(a)\Le \varphi(b)$,
since $\dim \hh\geqslant 1$. This completes the proof.
\end{proof}
In fact, as we will see, every increasing Borel function $\varphi\colon
\rbb^{\kappa}\rightarrow \rbb$ preserves multidimensional spectral order. In order to prove this, we will need a few lemmata.
\begin{lemma}\label{lemmiara}
  Let $\mathcal{M}$ be a $\sigma$-algebra on a set
$X$ and let $E$ be a spectral measure on
$\mathcal{M}$. If
$\{V_n\}_{n=1}^{\infty}\subseteq\mathcal{M}$, then
\begin{equation}\label{miara}
 E(\bigcup_{n=1}^{\infty}V_n)=\bigvee_{n=1}^{\infty}E(V_n).
\end{equation}
\end{lemma}
\begin{proof} Set $V_\infty:=\bigcup_{n=1}^{\infty}V_n$ and $M_k:=E(V_k)\hh$, for $k\in\nbb\cup \{\infty\}$. It suffices to show that $P_{M_\infty}=E(V_\infty)\Le \bigvee_{n=1}^{\infty}E(V_n)$,
 since $E(V_n)\Le E(V_\infty)$ for every $n\in\nbb$. By equation \eqref{sup1}, we have $\bigvee_{n=1}^{\infty}E(V_n)=\bigvee_{n=1}^{\infty}P_{M_n}=P_{\bigvee_{n=1}^{\infty}M_n}$.
 Thus we need to show that
 $M_\infty\subseteq\bigvee_{n=1}^{\infty}M_n$. Take
 $h\in M_\infty$ and let
\begin{equation*}
W_n:=\begin{cases}
V_1,& \textmd{for}\: n=1,\\
V_n\backslash(V_1\cup\ldots V_{n-1}),& \textmd{for}\: n>1.
\end{cases}
\end{equation*}
 It is obvious that
$V_\infty=\bigcup_{n=1}^{\infty}W_n$ and $ W_n\cap
W_m=\varnothing$ if $m\neq n$. Hence
\begin{align*} h=E(\bigcup_{n=1}^{\infty}W_n)h=\sum_{n=1}^{\infty}E(W_n)h\in\bigvee_{n=1}^{\infty}M_n.
\end{align*}
 This completes the proof.
\end{proof}
 \begin{lemma}\label{regular2}
 Let $E$ be a spectral measure on
$\borel(\rbb^{\kappa})$,
$\varOmega\in\borel(\rbb^{\kappa})$ and
$\iota\in\{1,\ldots,\kappa\}$. If $\varOmega\in
L(\rbb^\kappa,\Le_\iota)$, then
\begin{equation}\label{supremum}
E(\varOmega)=\bigvee\{ E(D)\colon
D\subseteq\varOmega,
D=\overline{D}\in L(\rbb^\kappa,\Le_\iota)\}.
\end{equation}
\end{lemma}
\begin{proof}
Let $E_0(\varOmega):=\bigvee\{ E(D)\colon
D\subseteq\varOmega,
D=\overline{D}\in L(\rbb^\kappa,\Le_\iota)\}$. It
is obvious that $E_0(\varOmega)\Le E(\varOmega)$. For
the opposite inequality, take any compact set
$K\subseteq\varOmega$. Applying Proposition \ref{cc}  and Proposition
\ref{cwiartka} (ii) we get
that $\overline{\downarrow_\iota K}=\downarrow_\iota K\subseteq \downarrow_\iota
\varOmega=\varOmega$ and
\begin{align*}
E(K)\Le E(\downarrow_\iota K)\Le E_0(\varOmega).
\end{align*}
This and corollary \ref{regular1} imply that
\begin{align*}
E(\varOmega)=\bigvee\{
E(K)\colon K\subseteq\varOmega,
K \textmd{ is compact}\} \Le E_0(\varOmega),
\end{align*}
which completes the proof of equality \eqref{supremum}.
\end{proof}
 We need some additional notation.
Let $\iota\in\nbb$ such that $\iota<\kappa$. If  $\mathbf{C}=(C_1,\ldots, C_{\kappa})\in\shk$, then we define $\mathbf{C}':=(C_1,\ldots, C_{\iota})\in\mathcal{S}_c(\hh,\iota)$ and $\mathbf{C}'':=(C_{\iota+1},\ldots, C_{\kappa})\in\mathcal{S}_c(\hh,\kappa-\iota)$. Applying spectral theorem for finitely many commuting selfadjoint operators and measure transport theorem  \cite[Theorem 5.4.10.]{Bir-Sol}, we can deduce that $E_{\mathbf{C}'}$ and $E_{\mathbf{C}''}$ commute and
\begin{equation}\label{iloczyn_C}
E_\mathbf{C}(\varOmega'\times\varOmega'')=E_{\mathbf{C}'}(\varOmega')E_{\mathbf{C}''}(\varOmega''),\, \quad \varOmega'\in\borel(\rbb^\iota),\, \varOmega''\in\borel(\rbb^{\kappa-\iota}).
\end{equation}
\begin{lemma}\label{lematpodst}
 Let $\A=(A_1,\ldots, A_{\kappa}),\B=(B_1,\ldots, B_{\kappa})\in\shk$ and let $\iota\in\{1,\ldots,\kappa\}$. Assume that $A_j=B_j$ for every $j=\iota+1,\ldots,\kappa$, whenever $\iota<\kappa$. If $\A\Lec\B$, then
\begin{equation}\label{nier}
E_{\B}(\varOmega)\Le E_{\A}(\varOmega)
\end{equation}
for every $\varOmega\in\borel(\rbb^\kappa)\cap L(\rbb^\kappa,\Le_\iota)$.
\end{lemma}
\begin{proof} The proof of inequality \eqref{nier} will be divided into several steps. We are going to deal only with the case $\iota<\kappa$, since the proof in the case $\iota=\kappa$ is similar.

{\em Step} 1. Inequality
\eqref{nier} holds for $\varOmega=(-\infty,x']\times \varOmega''$,
where $x'\in\rbb^\iota$ and
$\varOmega''\in\borel(\rbb^{\kappa-\iota})$.

Equation \eqref{iloczyn_C}, Lemma \ref{lemmiara}, and inequality $\A\Lec\B$ imply that
\begin{align*}E_{\B'}((-\infty,x'])&=E_{\B}((-\infty,x']\times \rbb^{\kappa-\iota})\\
&=\bigvee_{q''\in\mathbb{Q}^{\kappa-\iota}}E_{\B}((-\infty,x']\times (-\infty,q''])\\
&\Le\bigvee_{q''\in\mathbb{Q}^{\kappa-\iota}}E_{\A}((-\infty,x']\times (-\infty,q''])=E_{\A'}((-\infty,x']).
\end{align*}
By assumption, $E_{\B''}=E_{\A''}$. Therefore,
\begin{align*}
E_{\B}((-\infty,x']\times \varOmega'')&=E_{\B'}((-\infty,x'])E_{\B''}(\varOmega'')=E_{\B'}((-\infty,x'])E_{\A''}(\varOmega'')\\
&\Le
E_{\A'}((-\infty,x'])E_{\A''}(\varOmega'')=E_{\A}((-\infty,x']\times \varOmega'').
\end{align*}

{\em Step} 2. Inequality \eqref{nier}
holds, if
$\varOmega=\bigcup_{n=1}^{\infty}(-\infty,x_n]\times
\varOmega''_n$, where $x_n\in\rbb^\iota$ and
$\varOmega''_n\in\borel(\rbb^{\kappa-\iota})$ for all
$n\in\nbb$.

Indeed, by Lemma \ref{lemmiara} and Step 1., we get that
\begin{align*}
 E_{\B}(\bigcup_{n=1}^{\infty}(-\infty,x_n]\times \varOmega''_n)&=\bigvee_{n=1}^{\infty}E_{\B}((-\infty,x_n]\times \varOmega''_n)\\
&\Le\bigvee_{n=1}^{\infty}E_{\A}((-\infty,x_n]\times \varOmega''_n)=E_{\A}(\bigcup_{n=1}^{\infty}(-\infty,x_n]\times \varOmega''_n).
\end{align*}

{\em Step} 3. Assume that $\varOmega$ is a lower set in
$(\rbb^\kappa,\Le_\iota)$. Let $\Vert \cdot\Vert=\Vert \cdot\Vert_m$, $m\in\nbb$, be a  norm on $\rbb^m$ given by $\Vert x\Vert_m=\sum_{j=1}^m\vert x_j\vert$ for every $x\in\rbb^m$.
    We are going to show that
  \begin{equation}\label{inkluzja}
  \varOmega\subseteq \bigcup\{(-\infty,q']\times
\mathbb{B}_{\kappa-\iota}\left(q'',\varepsilon\right)\colon
(q',q'')\in (\mathbb{Q}^{\iota}\times\mathbb{Q}^{\kappa-\iota})\cap\varOmega^{(\varepsilon)}\}\subseteq\varOmega^{(2\varepsilon)},
  \end{equation}
  for every $\varepsilon>0$, where $\mathbb{B}_{m}\left(x,r\right)$ denote open ball $\{y\in\rbb^m\colon \Vert y-x\Vert_m<r\}$ for $x\in\rbb^m$ and $r>0$.

   Let $x\in\varOmega$. By the density of $\mathbb{Q}^{\kappa}$, we find
$q\in\mathbb{B}_\kappa\left(x,\varepsilon\right)\cap\mathbb{Q}^{\kappa}$ such that $x<q$.
Hence $q\in\varOmega^{(\varepsilon)}$ and
$x\in(-\infty,q']\times
\mathbb{B}_{\kappa-\iota}\left(q'',\varepsilon\right)$. This proves the first
inclusion.

 To prove the second inclusion fix $q=(q',q'')\in (\mathbb{Q}^{\iota}\times\mathbb{Q}^{\kappa-\iota})\cap\varOmega^{(\varepsilon)}$. Then \begin{align*}
 \{q'\}\times\mathbb{B}_{\kappa-\iota}\left(q'',\varepsilon\right)\subseteq \mathbb{B}_{\kappa}\left(q,\varepsilon\right)\subseteq \varOmega^{(2\varepsilon)},
\end{align*}
since $d_{\varOmega}(q)\Le\varepsilon$. Hence, by Proposition \ref{cwiartka} and Proposition \ref{oepsilon},
 \begin{align*}
  (-\infty,q']\times
\mathbb{B}_{\kappa-\iota}\left(q'',\varepsilon\right)=\downarrow_\iota\big( \{q'\}\times\mathbb{B}_{\kappa-\iota}\left(q'',\varepsilon\right)\big)\subseteq \downarrow_\iota\varOmega^{(2\varepsilon)}= \varOmega^{(2\varepsilon)}.
 \end{align*}
This gives the second inclusion.

   {\em Step} 4. Inequality
\eqref{nier} holds whenever $\varOmega$ is a closed lower set in $(\rbb^\kappa,\Le_\iota)$.

Let $Q_{q,n}:=(-\infty,q']\times \mathbb{B}_{\kappa-\iota}\left(q'',\frac{1}{n}\right)$, where $q=(q',q'')\in (\mathbb{Q}^{\iota}\times\mathbb{Q}^{\kappa-\iota})\cap\varOmega^{(\frac{1}{n})}$ and $n\in\nbb$. By \eqref{inkluzja} and Step 2., we obtain that
   \begin{align*}
   E_{\bf{B}}(\varOmega)\Le E_{\bf{B}}(\bigcup
Q_{q,n})\Le E_{\bf{A}}(\bigcup Q_{q,n})\Le
E_{\bf{A}}(\varOmega^{(\frac{2}{n})}).
   \end{align*}
   Note that
$\{\varOmega^{(\frac{1}{n})}\}_{n=1}^{\infty}$ is a
decreasing family of sets (with respect to set
inclusion). What is more,
$\varOmega=\bigcap_{\epsilon\in(0,\infty)}\varOmega^{(\epsilon)}$,
since $\varOmega$ is closed.
    Hence
   \begin{align*}
   E_{\bf{B}}(\varOmega)\Le \textsc{sot-}\lim_{n\to\infty}
E_{\bf{A}}(\varOmega^{(\frac{2}{n})})=E_{\bf{A}}(\varOmega).
   \end{align*}

 {\em Step} 5. Application of Step 4. and  Lemma \ref{regular2} completes the proof of
the inequality \eqref{nier} for an arbitrary Borel lower set
$\varOmega$ in $(\rbb^\kappa,\Le_\iota)$.
 \end{proof}
\begin{thm}\label{nowe}
Let $\A=(A_1,\ldots, A_{\kappa}),\B=(B_1,\ldots,
B_{\kappa})\in\shk$, and $\varphi\colon\rbb^\kappa\to\rbb$ be a Borel function. Assume that $\bf{A}\Lec \bf{B}$. If
 \begin{enumerate}
 \item[{\em (a)}] the function $\varphi$ is increasing,
  \end{enumerate}
  or there exist $\iota\in\nbb$ such that $\iota<\kappa$, and $\varOmega\in\borel(\rbb^{\kappa-\iota})$ satisfying the following conditions
   \begin{enumerate}
 \item[{\em (a')}] $A_j=B_j$ for every $j=\iota+1,\ldots,\kappa$,
 \item[{\em (b')}] the function $\varphi$
 is $\iota$-increasing on $\rbb^\iota\times\varOmega$,
\item[{\em (c')}] $E_{\A''}(\rbb^{\kappa-\iota}\setminus\varOmega)=E_{\B''}(\rbb^{\kappa-\iota}\setminus\varOmega)=0$,
 \end{enumerate}
  then $\varphi(\bf{A})\Lec\varphi(\bf{B})$.
\end{thm}
\begin{proof} We are going to prove only the second case (i.e. $\iota<\kappa$), since similar
line of reasoning applies in the first case. First let us observe, that the set
\begin{equation*}
\varphi^{-1}((-\infty,x])\cup\left(\rbb^\iota\times(\rbb^{\kappa-\iota}\setminus\varOmega)\right)
\end{equation*}
is a lower set in $(\rbb^\kappa,\Le_\iota)$ for every
$x\in\rbb$ by the assumption (b').  Using measure
transport theorem,
equality \eqref{iloczyn_C}, and
(c'), we get that
\begin{align*}
E_{\varphi(\mathbf{C})}((-\infty,x])&=E_{\mathbf{C}}(\varphi^{-1}((-\infty,x]))\\
&=E_{\mathbf{C}}\big(\varphi^{-1}((-\infty,x])\cup\left(\rbb^\iota\times(\rbb^{\kappa-\iota}\setminus\varOmega)\right)\big),
\end{align*} where $\mathbf{C}=\A,\B$. Thus, by Lemma \ref{lematpodst},
\begin{align*}
 E_{\varphi(\B)}((-\infty,x])\Le E_{\varphi(\A)}((-\infty,x])
\end{align*}
for every $x\in\rbb$.
 This completes the proof.
\end{proof}
\begin{thm}\label{spekwiel}
 Let $\A=(A_1,\ldots, A_{\kappa}),\B=(B_1,\ldots, B_{\kappa})\in\shk$. Then the following conditions are equivalent:
\begin{enumerate}[{\em(i)}]
    \item $\bf{A}\Lec\bf{B}$,
    \item $\varphi(\bf{A})\Lec\varphi(\bf{B})$ for every increasing Borel function  $\varphi\colon\rbb^{\kappa}\rightarrow\rbb$,
    \item $A_j\Lec B_j$ for every $j=1,\ldots,\kappa$.
\end{enumerate}
\end{thm}
\begin{proof}
(i)$\Rightarrow$(ii) It follows directly from Theorem \ref{nowe}.

(ii)$\Rightarrow$(iii) Apply equation \eqref{wspolnamiara} for the function $\pi_j$, which is increasing and Borel for every $j=1,\ldots,\kappa$.

(iii)$\Rightarrow$(i) Fix an arbitrary $x=(x_1,\ldots,x_{\kappa})\in\rbb^\kappa$. By the definition of the multidimensional spectral order for selfadjoint operators and equality \eqref{miaraEa}, we get that
\begin{align*}
F_{\B}(x)&=E_{\B}((-\infty,x])
=E_{B_1}((-\infty,x_1])\ldots E_{B_{\kappa}}((-\infty,x_{\kappa}])\\
&=F_{B_1}(x_1)\ldots F_{B_{\kappa}}(x_{\kappa})
\Le F_{A_1}(x_1)\ldots F_{A_{\kappa}}(x_{\kappa})=F_{\A}(x),
\end{align*}
for every $x_1,\ldots, x_{\kappa}\in\rbb$. This completes the proof.
\end{proof}
The following corollary give us an alternative definition for the multidimensional spectral order (see \cite[Theorem 1.]{fujii-kasahara} for bounded selfadjoint operators and \cite[Theorem 6.5.]{Paneta2012} for unbounded selfadjoint operators) in terms of standard order ''$\Le$'' for bounded selfadjoint operators and the spectral integrals of increasing Borel functions.
\begin{cor}\label{def-eq}
 Let $\A,\B\in\shk$. Then the following conditions are equivalent:
\begin{enumerate}[{\em(i)}]
\item $\bf{A}\Lec\bf{B}$,
\item $\varphi(\bf{A})\Le\varphi(\bf{B})$ for every bounded increasing continuous function  $\varphi\colon\rbb^{\kappa}\rightarrow\rbb$,
\item $\varphi(\bf{A})\Le\varphi(\bf{B})$ for every bounded increasing Borel function $\varphi\colon\rbb^{\kappa}\rightarrow\rbb$.
\end{enumerate}
\end{cor}
\begin{proof}
(i)$\Rightarrow$(iii) Notice that $\varphi(\bf{A}),\varphi(\bf{B})\in\ogr{\hh}$ for every bounded Borel function $\varphi\colon\rbb^{\kappa}\rightarrow\rbb$ and apply Theorem \ref{spekwiel} and \cite[Proposition 6.3.]{Paneta2012}.

(iii)$\Rightarrow$(ii) It is obvious.

(ii)$\Rightarrow$(i) In view of Theorem \ref{spekwiel}, it is sufficient to show that $A_j\Lec B_j$ for every $j=1,\ldots,\kappa$. Fix $j\in\{1,\ldots\kappa\}$. Let $f\colon\rbb\rightarrow\rbb$ be an arbitrary continuous and bounded increasing function. In this case  $\varphi:=f\circ\pi_j\colon\rbb^{\kappa}\rightarrow\rbb$
is a continuous and bounded increasing function. By \cite[Lemma 6.5.2.]{Bir-Sol} and condition (ii) we get that
\begin{align*}
f(A_j)=(f\circ\pi_j)(\A)=\varphi(\A)\Le \varphi(\B)=(f\circ\pi_j)(\B)=f(B_j).
\end{align*}
 Since this inequality holds for every bounded continuous increasing function $f\colon\rbb\rightarrow\rbb$, we deduce that $A_j\Lec B_j$ by  \cite[Theorem 6.5.]{Paneta2012}.
\end{proof}
\begin{rem} Note that there is another proof of the implication (ii)$\Rightarrow$(i). Indeed, similarly as in the proof of the implication (ii)$\Rightarrow$(i) in \cite[Theorem 6.5.]{Paneta2012}, we obtain that
 \begin{equation*} \int_{\rbb^\kappa}\varphi(x)  \is{ E_\A(\D x)h}{h}\Le \int_{\rbb^\kappa}\varphi(x)  \is{ E_\B(\D x)h}{h}, \quad h\in\hh,
 \end{equation*} for every bounded increasing continuous function  $\varphi\colon\rbb^{\kappa}\rightarrow\rbb$. Thus, by Theorem \ref{3.1}, $\is{F_\B(x)h}{h}\Le \is{F_\A(x)h}{h}$ for every $x\in\rbb^\kappa$ and $h\in\hh$. This means that $\A\Lec\B$.
\end{rem}
Let $\A\in\shk$ and
 $\varphi=(\varphi_1,\ldots,\varphi_\iota)\colon\rbb^\kappa\to\rbb^\iota$ be a Borel function. Set
\begin{equation}\label{phiA}
 \varphi(\A):=(\varphi_1(\A),\ldots,\varphi_{\iota}(\A)).
\end{equation}
It is evident that $\varphi(\A)\in\mathcal{S}_c(\hh,\iota)$.
\begin{cor}\label{monot1}
Let $\kappa,\iota\in\nbb$.
     Suppose that $\A,\B\in\shk$ and $\A\Lec\B$. If $\varphi\colon\rbb^\kappa\to\rbb^\iota$ is an increasing Borel function, then $\varphi(\A)\Lec\varphi(\B)$.
\end{cor}
\begin{proof} First let us observe, that the functions $\varphi_j:=\pi_j\circ \varphi$, where $j=1,\ldots,\iota$,  are increasing, since $\varphi$ is increasing.
  Applying Theorem \ref{spekwiel} to $\varphi_j$ we get that $\varphi_j(\A)\Lec\varphi_j(\B)$ for $j=1,\ldots,\iota$. By Theorem \ref{spekwiel}, this is equivalent to the inequality $\varphi(\A)\Lec\varphi(\B)$.
\end{proof}

 Now let us turn our attention to two questions concerning the multidimensional spectral order. The first problem is whether the partially ordered sets $(\shk,\Lec)$ and $(\B_{cs}(\hh,\kappa),\Lec)$ are a conditionally complete lattice if $\kappa>1$.
The second question that we study is whether spectral order satisfies some kind of vector order properties.

The first question is answered in the negative by the following example. In fact, the sets $(\shk,\Lec)$ and $(\B_{cs}(\hh,\kappa),\Lec)$ are not even a lattice for $\kappa>1$.
 \begin{exa}\label{lattice} Let $\hh=\cbb^3$ be a Hilbert space with the standard orthonormal basis $\{(1,0,0),(0,1,0),(0,0,1)\}$. Then $(\mathcal{S}(\hh,2),\Lec)$ and $(\B_{cs}(\hh,2),\Lec)$ are not a lattice. Indeed, let $M_1$, $M_2$, $N_1$, and $N_2$ are defined to be the linear subspaces of $\cbb^3$ of the form
 \begin{align*}
 M_1=\cbb^2\times\{0\},\quad M_2=\{(z_1,z_2,z_3)\in\cbb^3\colon z_1=z_2\},\\
\quad N_1=\{
0\}\times\cbb^2 \quad \text{ and } \quad N_2=\{(z_1,z_2,z_3)\in\cbb^3\colon z_2=z_3\}.\end{align*}
 Define $\A=(A_1,A_2)$ and $\B=(B_1,B_2)$ by
 \begin{align*}
 A_1&=P_{M_1}=\left[
                  \begin{array}{ccc}
                    1 & 0 & 0 \\
                    0 & 1 & 0 \\
                    0 & 0 & 0 \\
                  \end{array}
                \right], \quad  A_2=P_{M_2}=\left[
                  \begin{array}{ccc}
                    \frac 12 & \frac 12 & 0 \\
                    \frac 12 & \frac 12 & 0 \\
                    0 & 0 & 1 \\
                  \end{array}
                \right],\\
                B_1&=P_{N_1}=\left[
                  \begin{array}{ccc}
                    0 & 0 & 0 \\
                    0 & 1 & 0 \\
                    0 & 0 & 1 \\
                  \end{array}
                \right], \quad \textmd{and}\quad  B_2=P_{N_2}=\left[
                  \begin{array}{ccc}
                                      1 & 0 & 0 \\
                    0&\frac 12 & \frac 12 \\
                    0&\frac 12 & \frac 12  \\
                  \end{array}
                \right].
 \end{align*} It is obvious that $\A, \B\in \B_{cs}(\hh,2)$. Suppose that $\mathbf{C}=(C_1,C_2)$ is the infimum of $\{\A,\B\}$ in $(\B_{cs}(\hh,2),\Lec)$. Then, by Theorem \ref{spekwiel} and the definition of infimum, we obtain that
 \begin{align}\label{1}  C_j\Lec \inf\{P_{M_j},P_{N_j}\}=P_{M_j\cap N_j} \text{ for } j=1,2.
 \end{align}On the other hand, $(P_{M_1\cap N_1},0),\, (0,P_{M_2\cap N_2})\in\B_{cs}(\hh,2)$. Moreover, the set $\{\A,\B\}$ is bounded below by  $(P_{M_1\cap N_1},0)$ and $(0,P_{M_2\cap N_2})$ with respect to the multidimensional spectral order. Hence, by the definition of infimum and Theorem \ref{spekwiel}, we deduce that
  \begin{align}\label{2} P_{M_j\cap N_j} \Lec C_j \text{ for } j=1,2.
 \end{align}
 Inequalities \eqref{1} and \eqref{2} imply that $\mathbf{C}=(P_{M_1\cap N_1},P_{M_2\cap N_2})\notin B_{cs}(\hh,2)$ since \begin{align*} P_{M_1\cap N_1}=\left[
 \begin{array}{ccc}
 0& 0 & 0 \\
                                                                                                                                                                      0&1 & 0 \\
                                                                                                                                                                      0& 0 & 0 \\
                                                                                                                                                                   \end{array}                                                                                                                                                          \right]
 \textmd{ and } P_{M_2\cap N_2}=\frac{1}{3}\left[
                                  \begin{array}{ccc}
                                    1 & 1 & 1 \\
                                    1 & 1 & 1 \\
                                    1 & 1 & 1 \\
                                  \end{array}
                                \right]
 \end{align*} do not commute. Thus $(\B_{cs}(\hh,2),\Lec)$ is not a lattice. The same argument shows that $(\mathcal{S}_c(\hh,2),\Lec)$ is not a lattice too.
 \end{exa}
Now we are going to consider the second question. As shown by Olson (cf. \cite{Olson1971}) the relation
$0\Lec A-B$ may not imply $A\Lec B$ for arbitrary bounded selfadjoint operators $A,B$ unless $A$ and $B$ commute. Thus the spectral order is not a vector order. It is known also that $A\Lec B$ imply $\lambda A\Lec \lambda B$ whenever $\lambda\in[0,\infty)$. It appears that this properties can be strengthened, which is the consequence of Theorem \ref{spekwiel} and Theorem \ref{nowe}.
\begin{cor}\label{summno}
  Let $\A=(A_1,A_2),\, \B=(B_1,B_2)\in \mathcal{S}_c(\hh,2)$. Assume that $A_1\Lec B_1$ and $A_2\Lec B_2$. Then
\begin{equation}\label{suma}\overline{A_1+A_2}\Lec \overline{B_1+B_2}.
\end{equation}
\end{cor}
\begin{proof} Note that $\A\Lec\B$ by assumptions and Theorem \ref{spekwiel}. Let $\varphi\colon\rbb^2\rightarrow\rbb$ be a function  defined by $\varphi(x_1,x_2):=x_1+x_2$ for $x_1,x_2\in\rbb$. By \cite[Theorem 5.4.7]{Bir-Sol}
and Theorem \ref{spekwiel} applied to the function $\varphi$, which is increasing, we get that
\begin{align*}
\overline{A_1+A_2}=\varphi(A_1,A_2)\Lec \varphi(B_1,B_2)=\overline{B_1+B_2}.
\end{align*}This completes the proof of \eqref{suma}.
\end{proof}

\begin{cor}\label{iloczyn} Suppose that $A,B,C$ are selfadjoint operators in $\hh$. Assume also that $C$ is positive and commutes with $A$ and $B$. If $A\Lec B$, then
\begin{equation*}\overline{AC}\Lec \overline{BC}.
\end{equation*}
\end{cor}
\begin{proof} Let $\A=(A,C)$ and $\B=(B,C)$. By Theorem \ref{spekwiel} (iii), $\A\Lec \B$. Take a function $\psi\colon\rbb^2\rightarrow\rbb$ given by $\psi(x_1,x_2):=x_1\cdot x_2$ for  $x_1,x_2\in\rbb$. It is easily seen that the function $\psi$ is 1-increasing on $\rbb\times[0,\infty)$ and $E_C((-\infty,0))=0$. Thus, by Theorem \ref{nowe} and  \cite[Theorem 5.4.7]{Bir-Sol}, we obtain that
\begin{align*}
\overline{AC}=\psi(A,C)\Lec \psi(B,C)=\overline{BC}.
\end{align*}
\end{proof}
We close this section with the example of an order for normal operators.  In \cite{Brenna2012} Brenna and Flori defined  spectral order for bounded normal operators. The definition of this order can be easily adopted also in the case of unbounded normal operators.
\begin{exa}\label{normalny1} Let $T$ be a normal operator in $\hh$, i.e. densely defined operator such that $\dz{T}=\dz{T^*}$ and $\|Th\|=\|T^*h\|$ for every $h\in\dz{T}$. The notation $\mathcal{N}(\hh)$ will mean the set of all normal operators in $\hh$ and $\mathbf{B}_n(\hh):=\mathcal{N}(\hh)\cap\ogr{\hh}$. Denote by $E_T$ the spectral measure of $T$ on $\borel(\cbb)$, where $\cbb\simeq \rbb^2$ via $\rho(x,y)=x+iy$, $x,y\in\rbb$. Consider the real and imaginary part of $T$ defined by
\begin{align*} \mathrm{Re}\, T=\int_{\cbb}\frac{z+\bar{z}}{2}E_T(dz), \quad
\mathrm{Im}\, T=\int_{\cbb}\frac{z-\bar{z}}{2i}E_T(dz).
\end{align*}
Then
\begin{align*} \mathrm{Re}\, T=\overline{\left(\frac{T+T^*}{2}\right)},\quad \mathrm{Im}\, T=\overline{\left(\frac{T-T^*}{2i}\right)}, \quad \textrm{and } T=\mathrm{Re}\, T+i\mathrm{Im}\, T.
\end{align*}
Note that $\mathrm{Re}\, T$ and $\mathrm{Im}\, T$ are selfadjoint. Moreover, from the \cite[Theorem 5.4.10.]{Bir-Sol} we derive
\begin{align*}
E_{\mathrm{Re}\, T}(\Delta)=E_T(\rho(\Delta\times\rbb)) \textrm{ and } E_{\mathrm{Im}\, T}(\Delta)=E_T(\rho(\rbb\times\Delta)), \quad \Delta\in\borel(\rbb).
\end{align*}
Therefore $(\mathrm{Re}\, T,\mathrm{Im}\, T)\in\mathcal{S}_c(\hh,2)$.
Now, let $T_1,T_2\in\mathcal{N}(\hh)$. Then we write
\begin{equation*}
T_1\Lec T_2 \textrm{ if and only if } F_{\mathrm{Re}\, T_2}(x)F_{\mathrm{Im}\, T_2}(y)\Le F_{\mathrm{Re}\, T_1}(x)F_{\mathrm{Im}\, T_1}(y), \, (x,y)\in\rbb^2.
\end{equation*}
Applying equations \eqref{miaraEa} and \eqref{fa} we can rewrite this as
\begin{equation}\label{normalny}
T_1\Lec T_2 \textrm{ if and only if } F_{(\mathrm{Re}\, T_2,\mathrm{Im}\, T_2)}(x,y)\Le F_{(\mathrm{Re}\, T_1,\mathrm{Im}\, T_1)}(x,y), \, (x,y)\in\rbb^2.
\end{equation}
\end{exa}
The following propositions shows that the spectral order for pairs of commuting selfadjoint operators and the spectral order for normal operators are isomorphic in the category of partially ordered sets.
\begin{pro}\label{normalny2}The mappings $\Psi$ and $\Psi_0$ are order isomorphisms, where $\Psi\colon(\mathcal{N}(\hh),\Lec)\ni T\to (\mathrm{Re}\, T,\mathrm{Im}\, T)\in(\mathcal{S}_c(\hh,2),\Lec)$ and $\Psi_0\colon(\mathbf{B}_n(\hh),\Lec)\ni T\to(\mathrm{Re}\, T,\mathrm{Im}\, T)\in(\mathbf{B}_{cs}(\hh,2),\Lec)$.\end{pro}
\begin{proof}
Let $\A=(A_1,A_2)\in\mathcal{S}_c(\hh,2)$. Define the measure $E_\A\circ\rho^{-1}$ on $\borel(\cbb)$ by $E_\A\circ\rho^{-1}(\Delta)=E_\A(\rho^{-1}(\Delta))$, $\Delta\in\borel(\cbb)$. Then $T_{\A}=\int_{\cbb}z E_{\A}\circ\rho^{-1}(\D z)$ is a normal operator in $\hh$ such that $\mathrm{Re}\, T_\A=A_1$ and $\mathrm{Im}\, T_{\A}=A_2$.
Indeed, we have
\begin{align*}
E_{\mathrm{Re}\, T_\A}(\Delta)=E_{T_\A}(\rho(\Delta\times\rbb))=E_{\A}\circ\rho^{-1}(\rho(\Delta\times\rbb))=E_{A_1}(\Delta),\, \Delta\in\borel(\rbb),
\end{align*}
and
\begin{align*}
E_{\mathrm{Im}\, T_\A}(\Delta)=E_{T_\A}(\rho(\rbb\times\Delta))=E_{\A}\circ\rho^{-1}(\rho(\rbb\times\Delta))=E_{A_2}(\Delta),\, \Delta\in\borel(\rbb).
\end{align*}
In partcular, $\Psi$ is bijective. According to the definition of the multidimensional spectral order and \eqref{normalny}, $\Psi$ is order preserving. At the end, note that $\Psi(\mathbf{B}_n(\hh))=\mathbf{B}_{cs}(\hh,2)$, which completes the proof.
\end{proof}

\section{Joint bounded vectors} The aim of this section is to describe the joint resolution of the identity of $\A\in\shk$ in terms of joint bounded vectors. The main result
of this section Proposition \ref{wektogr} will be used in Section \ref{kolejny} to characterize  multidimensional spectral order in the case of positive operators.

First we are going to recall the definition of joint bounded vectors. Let $x=(x_1,\ldots,
x_{\kappa})\in\rbb^{\kappa}$ and $\alpha\in[0,\infty)^\kappa$. Set $x^\alpha:=x_1^{\alpha_1}\cdot\ldots\cdot
x_{\kappa}^{\alpha_{\kappa}}$ provided that $x_j^{\alpha_j}$ are well-defined for every $j=1,\ldots,\kappa$. Define functions $\varphi_{\alpha}\colon\rbb^\kappa\to\rbb$, $\alpha\in\zbb^\kappa_+$, and $\psi_{\beta}\colon\rbb^\kappa\to\rbb$, $\beta\in[0,\infty)^\kappa$, by the following formulas
\begin{align}\label{psibeta}
\varphi_{\alpha}(x)=x^{\alpha} \text{ and } \psi_\beta(x)=\vert
x_1\vert^{\beta_1}\cdot\ldots\cdot \vert
x_{\kappa}\vert^{\beta_{\kappa}}\chi_{[0,\infty)^\kappa}(x), \quad x\in\rbb^\kappa.
\end{align}
 The monomial $\A^\alpha$ is defined by
 \begin{align*}  \A^\alpha=
 \varphi_{\alpha}(\A) \text{ for }
\alpha\in\zbb_+^\kappa.
 \end{align*}
 Similarly, the fractional power $\A^\alpha$ is given by
 \begin{align*}  \A^\alpha=  \psi_\alpha(\A)& \textmd{ if } \supp E_\A\subseteq [0,\infty)^{\kappa}\textmd{ and }\alpha\in[0,\infty)^\kappa.
 \end{align*}
 Note that $\psi_\alpha(\A)=\varphi_{\alpha}(\A)$,
provided that $\alpha\in\zbb_+^\kappa$ and $\supp
E_\A\subseteq [0,\infty)^{\kappa}$.

Let $\A\in\shk$. We define the following sets
\begin{align*}
\dzn{\A}&:=\bigcap_{\alpha\in\zbb_+^{\kappa}}\dz{\A^\alpha},\\
\bscr_{a}(\A)&:= \bigcup_{\substack{c\in\rbb\\
c > 0}} \;
\left\lbrace h\in \dzn{\A}\colon \Vert \A^\alpha h\Vert\Le ca^\alpha \textmd{ for every}\:
\alpha\in\zbb_+^{\kappa} \right\rbrace
\text{ for } a\in
[0,\infty)^{\kappa},\\
\text{and }
\bscr(\A)&:=\bigcup_{a\in
(0,\infty)^{\kappa}}\bscr_{a}(\A).
\end{align*}
 We say that $h\in\hh$ is a \textit{joint
bounded vector} of $\A$ if
$h\in\bscr(\A)$. More information on joint bounded vectors can be found in \cite{sam}.
For the sake of completeness, Proposition \ref{wektogr} will be preceded by the following two Lemmata (see also \cite[Theorem 1.13.]{sam}).
\begin{lemma}\label{dodatnia} If $\A=(A_1,\ldots,A_{\kappa})\in\shk$, then
\begin{equation*} \supp E_{\A}\subseteq[0,\infty)^{\kappa} \text{ if and only if } A_j \text{ is positive for } j=1,\ldots,\kappa.
\end{equation*}
\end{lemma}
\begin{proof} Applying equalities \eqref{eapij} and \eqref{miara} we obtain that
\begin{align*}
\bigvee_{j=1}^\kappa E_{A_j}((-\infty,0))=\bigvee_{j=1}^\kappa E_{\A}(\rbb\times\ldots\times\underbrace{(-\infty,0)}_j\times\ldots\times\rbb)=E_{\A}(\rbb^\kappa\setminus[0,\infty)^\kappa),
\end{align*}
what proves the claim.
\end{proof}
If $\supp E_{\A}\subseteq[0,\infty)^{\kappa}$, then we say
that $\A$ is \textit{positive}.

\begin{lemma}\label{dnieskonczony} Let $\A=(A_1,\ldots,A_{\kappa})\in\shk$. Then
\begin{equation*}
\dzn{\A}=\bigcap_{j=1}^\kappa\dzn{A_j}.
\end{equation*}
Moreover, if $\A$ is positive, then
\begin{equation}\label{ajn1}
\dzn{\A}=\bigcap_{\alpha\in[0,\infty)^{\kappa}}\dz{\A^\alpha}.
\end{equation}
\end{lemma}
\begin{proof} Let
$D:=\displaystyle\bigcap_{j=1}^\kappa\dzn{A_j}$.  We have to show that $D\subseteq \dzn{\A}$ since
 inclusion $\dzn{\A}\subseteq D$ is obvious.

Let $h\in D$ and $\alpha\in\zbb_+^\kappa$. Consider a
measure $\mu_h$ on $\borel(\rbb^\kappa)$ defined by
$\mu_h(\sigma):=\is{E_\A(\sigma)h}{h}$ for
$\sigma\in\borel(\rbb^\kappa)$. Using the definition of
$D$ we derive that $\int_{\rbb^\kappa}x_j^{2\kappa\alpha_j}\D\mu_h(x)<\infty$ for
$j=1,\ldots,\kappa$. This and \cite[Corollary
2.11.5]{bogachev} imply that
$\int_{\rbb^\kappa}x^{2\alpha}\D\mu_h(x)<\infty$. Thus $h\in
\dz{\A^\alpha}$ for every $\alpha\in\zbb_+^\kappa$.
Consequently, $h\in \dzn{\A}$.

 For the ''moreover'' part, consider only the inclusion $\subseteq$ since the inclusion $\supseteq$ is evident.
Let $h\in \dzn{\A}$ and $\alpha\in[0,\infty)^\kappa$. Without loss of generality we may assume that $\Vert h\Vert=1$. Take $\beta\in\nbb^\kappa$ such that $\kappa\alpha\Le\beta$. By assumption, $\int_{[0,\infty)^\kappa}x_j^{2\beta_j}\D \mu_h(x)<\infty$. Applying \cite[Corollary 2.11.5]{bogachev} and Jensen inequality, we obtain that
  \begin{align*} \left(\int_{[0,\infty)^\kappa} x^{2\alpha}\D \mu_h(x)\right)^{\frac{1}{2}}&\Le \prod_{j=1}^\kappa\left(\int_{[0,\infty)^\kappa}( x_j^{\alpha_j})^{2\kappa}\D \mu_h(x)\right)^{\frac{1}{2\kappa}}\\ &\Le \prod_{j=1}^\kappa\left(\int_{[0,\infty)^\kappa}x_j^{2\beta_j}\D \mu_h(x)\right)^{\frac{\alpha_j}{2\beta_j}}<\infty.
  \end{align*}  Hence $h\in
\bigcap_{\alpha\in[0,\infty)^{\kappa}}\dz{\A^\alpha}$, which completes proof.
\end{proof}

 In the sequel we will use the following notation
$\dscr^{\Lambda}(\A):=\bigcap_{\alpha\in\Lambda}\dz{\A^\alpha}$ for
$\Lambda\subseteq[0,\infty)^\kappa$ and
$\vert\alpha\vert\index{$\vert\alpha\vert$\:}:=\alpha_1+\ldots+\alpha_{\kappa}$
for
$\alpha=(\alpha_1,\ldots,\alpha_{\kappa})\in[0,\infty)^{\kappa}$.
\begin{pro}\label{wektogr} Let $\A\in\shk$ be positive, $h\in\hh$, and $a\in [0,\infty)^{\kappa}$. Assume that $\Lambda\subseteq[0,\infty)^{\kappa}$ satisfies the condition
\begin{equation}\label{lambdasup}
\sup_{\alpha\in\Lambda}\frac{\alpha_j}{1+|\alpha|}=1, \quad j=1,\ldots,\kappa.
\end{equation} Then the following conditions are equivalent:
\begin{enumerate}[{\em(i)}]
\item $h\in\bscr_{a}(\A)$,
\item $h\in\ob{F_{\A}(a)}$,
  \item $h\in\dzn{\A}$ and there exists real number $c>0$ such that
\begin{equation*} \Vert \A^\alpha
h\Vert\Le ca^{\alpha}\: \textmd{for every}\:
\alpha\in[0,\infty)^\kappa,
\end{equation*}
\end{enumerate}
  Moreover, if $h\in\dscr^{\Lambda}(\A)\cap \left(\bigcup_{j=1}^{\kappa}\jd{A_j}\right)^\perp$ and there exists real number $c>0$ such that
\begin{equation}\label{dlalambda} \Vert \A^\alpha
h\Vert\Le ca^{\alpha}\: \textmd{for every}\:
\alpha\in\Lambda,\end{equation} then $h\in\ob{F_{\A}(a)}$.

\end{pro}
\begin{proof} Let $\mu_h(\sigma):=\is{E_{\A}(\sigma)h}{h}$
for every $\sigma\in\borel(\rbb^{\kappa})$.

(i)$\Rightarrow$(ii) By Lemma \ref{dnieskonczony},
$h\in\dzn{A_j}$ for every
$j=1,\ldots,\kappa$. At the same time $\Vert
A_j^nh\Vert=\Vert
\A^{ne_j}h\Vert\Le ca_j^n$ for every $n\in\zbb_+$. This
implies that $h\in \bscr_{a_j}(A_j)$. Hence, by \cite[Proposition 5.1.]{Paneta2012},
$h\in\ob{F_{A_j}(a_j)}$ for every $j=1,\ldots,\kappa$.
Thus $h\in\ob{F_{\A}(a)}$, since
$F_{\A}(a)=F_{A_1}(a_1)\ldots
F_{A_{\kappa}}(a_{\kappa})$.

(ii)$\Rightarrow$(iii) Let
$\alpha\in[0,\infty)^{\kappa}$. Then, by (ii) and positivity of
$\A$, we obtain that $\supp
\mu_h\subseteq[0,a]$ and
\begin{align*}
\int_{[0,\infty)^{\kappa}} x^{2\alpha}\,
\D\mu_h(x)=\int_{[0,a]}
x^{2\alpha}\, \D\mu_h(x)\Le a^{2\alpha}\Vert
h\Vert^2<\infty.
\end{align*}
Thus $h\in\dz{\A^\alpha}$ and
\begin{align*}
\Vert \A^\alpha h\Vert=\Big(\int_{[0,\infty)^{\kappa}}
x^{2\alpha}\, d\mu_h(x)\Big)^{\frac{1}{2}}\Le
a^{\alpha}\Vert h\Vert.
\end{align*}

(iii)$\Rightarrow$(i) It is obvious.

 For the ''moreover'' part, assume that
$h\in\dscr^{\Lambda}(\A)\cap \left(\bigcup_{j=1}^{\kappa}\jd{A_j}\right)^\perp$ satisfies equation \eqref{dlalambda}. Suppose that
$h\notin\ob{F_{\A}(a)}$. In particular, $\mu_h(\rbb^{\kappa}\backslash(-\infty,a])>0$.
Applying \cite[Theorem
   6.1.3]{Bir-Sol} and \eqref{eapij} we obtain that \begin{align*}E_\A(\rbb\times\ldots\times\underbrace{\{0\}}_j\times\ldots\times\rbb)h= E_{A_j}(\{0\})h=0,\quad j=1,\ldots,\kappa. \end{align*}
   What is more
   $E_{\A}(\rbb^{\kappa}\backslash[0,\infty)^{\kappa})=0$,
   since $\A$ is positive. Thus
\begin{equation}\label{muh}
\mu_h((0,\infty)^{\kappa}\backslash[0,a])=\mu_h(\rbb^{\kappa}\backslash(-\infty,a])>0.
\end{equation}
Put
$$D_{j,n}:=\left(\frac{1}{n},\infty\right)\times\ldots\times\underbrace{\left(a_j+\frac{1}{n},\infty\right)}_j\times\ldots\times\left(\frac{1}{n},\infty\right)$$
for $j\in\{1,\ldots,\kappa\}$ and $n\in\nbb$. Then
\begin{equation*}
(0,\infty)^{\kappa}\backslash[0,a]=\bigcup_{j=1}^{\kappa}\bigcup_{n=1}^{\infty}D_{j,n}.
\end{equation*} This together with \eqref{muh} allow us to choose $j\in\{1,\ldots,\kappa\}$ and $n\in\nbb$ such that
\begin{equation*}
d:=\mu_h(D_{j,n})>0.
\end{equation*}
Fix $\alpha\in\Lambda$. Applying inequality
\eqref{dlalambda} we obtain that
\begin{align*}
dn^{-2\vert \alpha \vert}\left(na_j+1\right)^{2\alpha_j}&\Le \int_{D_{j,n}} x^{2\alpha} \, d\mu_h(x)\\
&\Le \int_{\rbb^{\kappa}} x^{2\alpha} \,
d\mu_h(x)=\Vert \A^\alpha h\Vert^2
\overset{\eqref{dlalambda}}{\Le}c^2a^{2\alpha}.
\end{align*}
Hence
\begin{equation}\label{gammaalpha}
d^{\frac{1}{2\vert
\alpha\vert}}n^{-1}\left(na_j+1\right)^{\frac{\alpha_j}{\vert
\alpha \vert}}\Le c^{\frac{1}{\vert \alpha
\vert}}a^{\frac{\alpha}{\vert \alpha \vert}},\:
\textmd{for every}\: \alpha\in\Lambda.
\end{equation}
By condition \eqref{lambdasup}, we can find a
sequence
$\{\alpha(m)\}_{m=1}^{\infty}\subseteq\Lambda$ such
that
\begin{equation*}
\lim_{m\to\infty}\vert \alpha(m)\vert=\infty \textmd{
and } \lim_{m\to\infty}\frac{\alpha_j(m)}{\vert
\alpha(m)\vert}=1,
\end{equation*}
 where
 $\alpha(m)=(\alpha_1(m),\ldots,\alpha_\kappa(m))$. Then inequality \eqref{gammaalpha} implies that
\begin{align*}
a_j+\frac{1}{n}&=\lim_{m\to\infty}d^{\frac{1}{2\vert
\alpha(m)\vert}}n^{-1}(na_j+1)^{\frac{\alpha_j(m)}{\vert
\alpha(m) \vert}}\\
&\Le\lim_{m\to\infty}c^{\frac{1}{\vert \alpha(m)
\vert}}a^{\frac{\alpha(m)}{\vert \alpha(m)
\vert}}=\lim_{m\to\infty}a_j^{\frac{\alpha_j(m)}{\vert
\alpha(m) \vert}}=a_j,
\end{align*} which is a contradiction. Thus $h\in\ob{F_{\A}(a)}$.
\end{proof}
Assume that $\A$ and $\Lambda$ satisfy the assumptions of Proposition \ref{wektogr}. The following example shows that without the condition \eqref{dlalambda}
it may happen that $h\in\dscr^{\Lambda}(\A)\cap \left(\bigcup_{j=1}^{\kappa}\jd{A_j}\right)^\perp$ but $h\notin\dzn{\A}$.
\begin{exa}
Let $\Lambda:=[0,\infty)\times(0,\infty)$ and $\hh=L^2([1,\infty),m)$, where $m$ is Lebesgue measure on Borel subsets of $[1,\infty)$. Define $\A=\left(M_{\varphi_1},M_{\varphi_2}\right)$, where the functions
$\varphi_1,\varphi_2\colon [1,\infty) \rightarrow \rbb$
are given by $\varphi_1(x)= x$ and
$\varphi_2(x)=e^{-x}$ for $x\in\rbb$. Then, by Lemma
\ref{fmphi}, $\A\in\mathcal{S}_c({\mathcal H},2)$. $\A$ is also positive since $\varphi_j(x)\Ge 0$
for every $x\in\rbb$ and $j=1,2$. Let
$h(x):=\frac{1}{ x}$ for $x\in[1,\infty)$. It is
evident that $h\in\hh$ and the set $\Lambda$ satisfies
the condition \eqref{lambdasup}.

We are going to show that $h\in\dscr^{\Lambda}(\A)\backslash\dzn{\A}$. If
$\alpha\in\Lambda$, then $\alpha_2>0$. Hence
\begin{align*}
\int_{[1,\infty)}\vert
(\varphi_1(x))^{\alpha_1}(\varphi_2(x))^{\alpha_2}h(x)\vert^2\:
\D m(x)=\int_{[1,\infty)} x^{2\alpha_1-2}e^{-2\alpha_2 x}\: \D m(x)<\infty.
\end{align*}
Thus, by Lemma \ref{fmphi} (iii),
$h\in\dz{M_{\varphi_1^{\alpha_1}\varphi_2^{\alpha_2}}}=\dz{\A^\alpha}$. On the other hand,
\begin{equation*}
\int_{[1,\infty)}\vert h(x)\varphi_1(x)\vert^2\:
\D m(x)=\int_{[1,\infty)}1\: \D m(x)=\infty,
\end{equation*}
which means that $h\notin\dz{A_1}\supseteq\dzn{\A}$, which completes the proof.
\end{exa}

\section{Monomials and multidimensional spectral order}\label{dziedzina}
Let $\A,\B\in\shk$. In this section we discuss the relations between inequality $\A\Lec\B$ and the domains of $\A^\alpha$ and $\B^\alpha$ for $\alpha\in\zbb^\kappa$. In Theorem \ref{indziedzin}, the appropriate condition  guaranteeing the inclusion $\dz{\B^{\alpha}}\subseteq\dz{\A^{\alpha}}$ will be formulated.


We begin our consideration by some facts about
positive and negative parts of selfadjoint
operators. Let $A$ be a selfadjoint operator in $\hh$. We denote the positive part and the negative part of $A$ by $A_+$ and $A_-$, respectively. By the definition,
\begin{equation*}
A_\pm=f_\pm(A),
\end{equation*}
where the functions $f_\pm\colon\rbb\to\rbb$ are given by
\begin{equation}\label{fpm}
f_{\pm}(x):=\frac{1}{2}(x\pm \vert x\vert),\quad x\in\rbb.
\end{equation} It is known that
 $A_+$ and $A_-$ are selfadjoint. Moreover, $A_+$ is positive.
Since the functions $f_+$ and $f_-$ are increasing, the following property holds.
\begin{lemma}
If $A$ and $B$ are selfadjoint operators in $\hh$ such that $A\Lec B$, then $A_+\Lec B_+$ and $A_-\Lec B_-$.
\end{lemma}

 Let $\varepsilon=(\varepsilon_1,\ldots,\varepsilon_{\kappa})\in\{-,+\}^{\kappa}$. The function $f_{\varepsilon_1}\times\ldots\times f_{\varepsilon_{\kappa}}\colon\rbb^{\kappa}\rightarrow\rbb^\kappa$, where $f_{\pm}$ are given by formulas \eqref{fpm}, will be dentoted by  $f_{\varepsilon}$. In particular,
\begin{equation}\label{pijfe}
\pi_j\circ f_{\varepsilon}=f_{\varepsilon_j}\circ \pi_j,\quad \varepsilon\in\{-,+\}^{\kappa},\, j=1,\ldots,\kappa.
\end{equation}
 In the sequel we will use the following multi-sign $\dagger=(\dagger_1,\ldots,\dagger_{\kappa})\in\{-,+\}^{\kappa}$, where
$\dagger_j=+$ for $j=1,\ldots,\kappa$.
 For $\mathbf{C}=(C_1,\ldots,C_{\kappa})\in\shk$, we define $\mathbf{C}_{\varepsilon}\in\shk$ by
\begin{equation}\label{defceps}\mathbf{C}_{\varepsilon}:=(f_{\varepsilon_1}(C_{1}),\ldots, f_{\varepsilon_{\kappa}}(C_{\kappa})).
\end{equation} Then, by \eqref{phiA}, \eqref{pijfe}, and \cite[Lemma 6.5.2.]{Bir-Sol}, we derive that
\begin{align*}
f_{\varepsilon}(\mathbf{C})=\big((\pi_j\circ f_{\varepsilon})(\mathbf{C})\big)_{j=1}^{\kappa}=
((f_{\varepsilon_j}\circ \pi_j)(\mathbf{C}))_{j=1}^{\kappa}
=
(f_{\varepsilon_j}(C_j))_{j=1}^{\kappa}=\mathbf{C}_{\varepsilon}.
\end{align*}
Moreover, by equality
 $f_{\varepsilon_j}(C_j)=(\pi_j\circ f_{\varepsilon})(\mathbf{C})$ and \cite[Theorem 5.4.10]{Bir-Sol}, we get that
\begin{equation*}
f_{\varepsilon_j}(C_j)=\int_{\rbb^{\kappa}}x_jE_{\mathbf{C}}\circ f_{\varepsilon}^{-1}(dx),\, j=1,\ldots,\kappa.
\end{equation*}Thus, by \cite[Theorem 6.5.1]{Bir-Sol},
 \begin{equation}\label{miaraepsilon}
E_{\mathbf{C}_{\varepsilon}}=E_{\mathbf{C}}\circ f_{\varepsilon}^{-1}.
\end{equation}
\begin{lemma}
If $\mathbf{C}\in\shk$, then
\begin{equation}\label{aepsilon}
\dz{\mathbf{C}^{\alpha}}=\bigcap_{\varepsilon\in\{-,+\}^{\kappa}}\dz{(\mathbf{C}_{\varepsilon})^{\alpha}}.
\end{equation}
\end{lemma}
\begin{proof}
 Set $\rbb_-:=(-\infty,0]$. If $\varepsilon=(\varepsilon_1,\ldots,\varepsilon_{\kappa})\in\{-,+\}^{\kappa}$, then we define $\rbb_{\varepsilon}^{\kappa}:=\rbb_{\varepsilon_1}\times \ldots\times\rbb_{\varepsilon_{\kappa}}$.

Let $h\in\hh$. Applying equality \eqref{miaraepsilon}
 we get that
\begin{align*}
\int_{\rbb_{\varepsilon}^{\kappa}} x^{2\alpha} \is{E_\mathbf{C}(\D x)h}{h}
&=\int_{\rbb^{\kappa}}(f_{\varepsilon}(x))^{2\alpha} \is{E_\mathbf{C}(\D x)h}{h}\\
&=\int_{\rbb^{\kappa}} x^{2\alpha} \is{E_{\mathbf{C}}\circ f_{\varepsilon}^{-1}(\D x)h}{h}\\
&\overset{\eqref{miaraepsilon}}{=}\int_{\rbb^{\kappa}} x^{2\alpha} \is{E_{\mathbf{C}_{\varepsilon}}(\D x)h}{h}.
\end{align*}
Thus
\begin{align*}
\int_{\rbb^{\kappa}} x^{2\alpha} \is{E_\mathbf{C}(\D x)h}{h}&\Le\sum_{\varepsilon\in\{-,+\}^{\kappa}}\int_{\rbb_{\varepsilon}^{\kappa}} x^{2\alpha} \is{E_\mathbf{C}(\D x)h}{h}\\
=\sum_{\varepsilon\in\{-,+\}^{\kappa}}\int_{\rbb^{\kappa}} x^{2\alpha} \is{E_{\mathbf{C}_{\varepsilon}}(\D x)h}{h}
&\Le 2^\kappa\int_{\rbb^{\kappa}} x^{2\alpha} \is{E_\mathbf{C}(\D x)h}{h},
\end{align*}
which gives \eqref{aepsilon}, since $h\in\hh$ was chosen arbitrarily.
\end{proof}
The following theorem answers question raised at the beginning of this section.
\begin{thm}\label{indziedzin}
 Let $\A,\B\in\shk$  and
let $\alpha\in\zbb_+^{\kappa}$. If $\A\Lec\B$ and
\begin{equation}\label{xabh}
(\A_{\varepsilon})^{\alpha}\in\ogr{\hh},\quad \varepsilon\in\{-,+\}^{\kappa}\backslash\{\dagger\},
\end{equation}
then
 \begin{equation}\label{dxalfa}
\dz{\B^{\alpha}}\subseteq \dz{\A^{\alpha}}.
\end{equation}
\end{thm}
\begin{proof}
 First observe that the function $f_{\varepsilon}$ is increasing. Hence, by Corollary \ref{monot1} and assumptions, we get that $\A_{\varepsilon}=f_{\varepsilon}(\A)\Lec f_{\varepsilon}(\B)=\B_{\varepsilon}$.

Now we will prove the inclusion \eqref{dxalfa}. It follows from the condition \eqref{xabh} that $\dz{(\A_{\varepsilon})^{\alpha}}=\hh$ for every $\varepsilon\neq\dagger$. This and the equality \eqref{aepsilon} imply that
\begin{equation}\label{alfaplus}
\dz{\A^{\alpha}}=\bigcap_{\varepsilon\in\{-,+\}^{\kappa}}\dz{(\A_{\varepsilon})^{\alpha}}=
\dz{(\A_{\dagger})^{\alpha}}.
\end{equation}
   Note that $\supp
E_{\A_{\dagger}}\cup\supp E_{\B_{\dagger}}\subseteq[0,\infty)^{\kappa}$
by Corollary \ref{dodatnia} and the definition of $\A_{\dagger}$ and $\B_{\dagger}$.
Hence, by Theorem \ref{spekwiel}, we have
$(\A_{\dagger})^{\alpha}=\psi_\alpha(\A_\dagger)\Lec \psi_\alpha(\B_\dagger)=(\B_{\dagger})^{\alpha}$, since the function
$\psi_\alpha$ is increasing.
Moreover, the operators $(\A_{\dagger})^{\alpha}$ and
$(\B_{\dagger})^{\alpha}$ are positive. Applying  \cite[Proposition
7.1.]{Paneta2012} we
infer that $\dz{(\B_{\dagger})^{\alpha}}\subseteq
\dz{(\A_{\dagger})^{\alpha}}$. Eventually, using equalities
\eqref{aepsilon} and \eqref{alfaplus}, we deduce that
\begin{align*}
\dz{\B^{\alpha}}=\bigcap_{\varepsilon\in\{-,+\}^{\kappa}}\dz{(\B_{\varepsilon})^{\alpha}}\subseteq\dz{(\B_{\dagger})^{\alpha}}
\subseteq\dz{(\A_{\dagger})^{\alpha}}=\dz{\A^{\alpha}},
\end{align*}
which completes the proof.
\end{proof}
\begin{rem}
If $\kappa=1$ and $\alpha=\alpha_1=1$, then Theorem \ref{indziedzin} implies the second part of \cite[Proposition 6.3.]{Paneta2012}. Indeed, the condition $A_-\in\ogr{\hh}$ holds if and only if $A$ is bounded from below.
\end{rem}
We are now going to discuss in more detail the condition \eqref{xabh}. It may happen that the inclusion \eqref{dxalfa} may not hold, even if the condition \eqref{xabh} does not hold for only one $\varepsilon\in\{-,+\}^\kappa\backslash\{\dagger\}$. This will be illustrated by Example \ref{example}. The example will be preceded by two propositions describing spectral order for multiplication operators.
\begin{pro}\label{przyk}  Let $(X,\mathcal{A},\mu)$ be a measure space. Assume that $\mu$ is $\sigma$-finite.
 If $\varphi,\psi\colon X\to\rbb$ are $\mathcal{A}$-measurable, then the following conditions are equivalent:
  \begin{enumerate}[{\em(i)}]
  \item $M_{\varphi}\Lec M_{\psi}$,
   \item $\is{M_{\varphi}h}{h}\Le \is{M_{\psi}h}{h}$ for every $h\in\dz{M_{\varphi}}\cap\dz{M_{\psi}}$,
    \item $\varphi\Le \psi$ a.e. $[\mu]$.
  \end{enumerate}
\end{pro}
\begin{proof}
 (i)$\Rightarrow$(ii) It follows from \cite[Proposition 6.3.]{Paneta2012}.

 (ii)$\Rightarrow$(iii) Suppose that for some $N\in\nbb$ there exists an $\mathcal{A}$-measurable set $\Delta\subset\{x\in X\colon -N<\psi(x)<\varphi(x)<N\}$  such that $0<\mu(\Delta)<\infty$. Then
 \begin{align*}\int_\Delta \psi\, \textrm{d}\mu<\int_\Delta \varphi\, \textrm{d}\mu.
 \end{align*}
 On the other hand, $h:=\chi_\Delta\in
\dz{M_{\varphi}}\cap\dz{M_{\psi}}$. Thus, by assumption,
\begin{align*}\int_\Delta \varphi\, \textrm{d}\mu\Le\int_\Delta \psi\, \textrm{d}\mu,
 \end{align*}
 which is a contradiction. Hence $\varphi\Le \psi$ a.e. $[\mu]$, since $\mu$ is $\sigma$-finite.

 (iii)$\Rightarrow$(i) Let $t\in\rbb$. Note that
 \begin{equation}\label{zawieranie} \psi^{-1}((-\infty,t])\subset \{x\in X\colon \psi(x)<\varphi(x)\}\cup\varphi^{-1}((-\infty,t]).\end{equation} By assumption, $\mu(\{x\in X\colon \psi(x)<\varphi(x)\})=0$. Thus, by \cite[Example 4.3.]{Schmudgen} and \eqref{zawieranie}, we obtain that
 \begin{align*}
  \is{F_{M_\psi}(t)h}{h}=&\int_X \chi_{\psi^{-1}((-\infty,t])}\vert h\vert^2\textrm{d}\mu\\ \overset{\eqref{zawieranie}}\Le& \int_X \chi_{\varphi^{-1}((-\infty,t])}\vert h\vert^2\textrm{d}\mu=\is{F_{M_\varphi}(t)h}{h},
 \end{align*}
 for every $h\in L^2(X,\mu)$. Hence $F_{M_\psi}(t)\Le F_{M_\varphi}(t)$ for every $t\in\rbb$, which completes the proof.
\end{proof}
For the sake of completeness, we include the following lemma.
\begin{lemma}\label{fmphi} Let $(X,\mathcal{A},\mu)$ be a measure space. Assume that
$\varphi_j\colon X\to\rbb$ is an $\mathcal{A}$-measurable function for every $j=1,\ldots,\kappa$. If $\varphi=(\varphi_1,\ldots,\varphi_{\kappa})\colon X\to \rbb^\kappa$ and $\mathbf{M}_\varphi:=(M_{\varphi_1},\ldots,M_{\varphi_\kappa})$, then
\begin{enumerate}[{\em(i)}]
\item $\mathbf{M}_\varphi\in\mathcal{S}_c(L^2(X,\mu),\kappa)$,
\item for every $\sigma\in\borel(\rbb^\kappa)$ and $h\in L^2(X,\mu)$
\begin{equation}\label{eam}E_{\mathbf{M}_\varphi}(\sigma)h=\chi_{\varphi^{-1}(\sigma)}h,
\end{equation}
\item
if $f\colon\rbb^\kappa\to\rbb$ is a Borel function, then
\begin{equation}\label{fam} f(\mathbf{M}_\varphi)=M_{ f\circ\varphi}.
\end{equation}
\end{enumerate}
\end{lemma}
\begin{proof} (i)  Let $\sigma,\tau\in\borel(\rbb)$ and $h\in L^2(X,\mu)$.
Applying equation \eqref{emphi} we obtain that
\begin{align*}
 E_{M_{\varphi_i}}(\sigma)E_{M_{\varphi_j}}(\tau)h=\chi_{\varphi_i^{-1}(\sigma)}\chi_{\varphi_j^{-1}(\tau)}h
=E_{M_{\varphi_j}}(\tau)E_{M_{\varphi_i}}(\sigma)h,\quad i,j=1,\ldots,\kappa.
\end{align*}
 Hence $\mathbf{M}_\varphi\in\mathcal{S}_c(L^2(X,\mu),\kappa)$.

(ii) Let us define $E(\sigma)h:=\chi_{\varphi^{-1}(\sigma)}h$ for $\sigma\in\borel(\rbb^\kappa)$ and  $h\in L^2(X,\mu)$. Take $\sigma=\sigma_1\times\ldots\times\sigma_\kappa$, where $ \sigma_1,\ldots,\sigma_\kappa\in\borel(\rbb)$. Then, by \eqref{emphi} and \eqref{miaraEa},
\begin{align*}
E_{\mathbf{M}_\varphi}(\sigma)h&=E_{M_{\varphi_1}}(\sigma_1)\ldots E_{M_{\varphi_\kappa}}(\sigma_\kappa)h\\&=\chi_{\varphi_1^{-1}(\sigma_1)}\ldots \chi_{\varphi_\kappa^{-1}(\sigma_\kappa)}h=\chi_{\varphi^{-1}(\sigma)}h=E(\sigma)h.
\end{align*}
Hence $E_{\mathbf{M}_\varphi}(\sigma)=E(\sigma)$ for every $\sigma\in\borel(\rbb^\kappa)$, since $E$ is a spectral measure on $\borel(\rbb^\kappa)$ and $(\borel(\rbb))^\kappa$ generates $\sigma$-algebra $\borel(\rbb^\kappa)$.

(iii) Using \cite[Theorem 5.4.10.]{Bir-Sol}transport theorem, \eqref{eam}, and \eqref{emphi}, we obtain
\begin{align*}
E_{f(\mathbf{M}_\varphi)}=E_{\mathbf{M}_\varphi}\circ f^{-1}\overset{\eqref{eam}}{=}E\circ f^{-1}
=E_{M_{ f\circ\varphi}}.
\end{align*}
This implies \eqref{fam}, since there is a one to one correspondence between spectral measures and selfadjoint operators.
\end{proof}
Combining Proposition \ref{przyk}, Lemma \ref{fmphi}, and Theorem \ref{spekwiel}, we obtain the following proposition.
\begin{pro}\label{mnoz}
  Let $(X,\mathcal{A},\mu)$ be a $\sigma$-finite measure space. Assume that
$\varphi_j,\psi_j\colon X\to\rbb$ are $\mathcal{A}$-measurable functions for every $j=1,\ldots,\kappa$. Consider $\varphi=(\varphi_1,\ldots,\varphi_{\kappa})\colon X\to \rbb^\kappa$ and $\psi=(\psi_1,\ldots,\psi_{\kappa})\colon X\to \rbb^\kappa$. Then
\begin{equation*}
  \mathbf{M}_\varphi\Lec \mathbf{M}_\psi \text{ if and only if } \varphi\Le\psi \text{ a.e. } [\mu].
\end{equation*}
\end{pro}
\begin{exa}\label{example}
Let $\hh=L^2(\rbb^{\kappa},m_\kappa)$, where $m_\kappa$ denotes Lebesgue measure on $\borel(\rbb^\kappa)$, and let $\varepsilon=(\varepsilon_1,\ldots,\varepsilon_{\kappa})\in\{-,+\}^{\kappa}\backslash\{\dagger\}$.
Fix $j_0\in\{1,\ldots,\kappa\}$ such that $\varepsilon_{j_0}=-$. Define the functions $\varphi_j, \psi_j\colon \rbb^{\kappa}\to\rbb$ for $j\in\{1,\ldots,\kappa\}$ by
\begin{equation*}
\varphi_j(x)=\psi_j(x):=x_j\chi_{\rbb^\kappa_{\varepsilon}}(x), \text{ for } x=(x_1,\ldots,x_\kappa)\in\rbb^{\kappa} \text{ and }j\neq j_0,
\end{equation*}
\begin{equation*}
\varphi_{j_0}(x):=-(x_{j_0}^2+1)\chi_{\rbb_{\varepsilon_{j_0}}}(x_{j_0}),\quad x=(x_1,\ldots,x_\kappa)\in\rbb^\kappa,
\end{equation*}
and
\begin{equation*}
\psi_{j_0}(x):=x_{j_0}\chi_{\rbb_{\varepsilon_{j_0}}}(x_{j_0}),\quad x=(x_1,\ldots,x_\kappa)\in\rbb^\kappa.
\end{equation*}
It is evident that  $\varphi_j$ and $\psi_j$ are Borel function for every $j=1,\ldots,\kappa$.
 Let $\A:=M_{\varphi}$ and $\B:=M_{\psi}$, where $\varphi=(\varphi_1,\ldots,\varphi_{\kappa})$ and $\psi=(\psi_1,\ldots,\psi_{\kappa})$. According to
 Lemma \ref{fmphi},  $\A,\B\in\shk$.

We are going  to show that $\A$ and  $\B$ satisfy the following conditions:
\begin{enumerate}[(i)]
\item $\A\Lec\B$,
\item $(\A_{\delta})^{\alpha}\in\ogr{\hh}$ for every $\delta\in\{-,+\}^\kappa\backslash\{\varepsilon\}$ and for all $\alpha\in\nbb^{\kappa}$,
\item $(\A_{\varepsilon})^{\alpha}\notin\ogr{\hh}$ for every $\alpha\in\nbb^{\kappa}$,
\item $\dz{\B^{\alpha}}\not\subseteq \dz{\A^{\alpha}}$ for every $\alpha\in\nbb^{\kappa}$.
\end{enumerate}

(i) This is the consequence of Proposition \ref{mnoz}, since  $\varphi_j(x)\Le\psi_j(x)$ for every $x\in\rbb^{\kappa}$ and $j=1,\ldots,\kappa$.

(ii) We can find $i\in\{1,\ldots,\kappa\}$ such that $\delta_{i}\neq\varepsilon_{i}$, since $\delta=(\delta_1,\ldots,\delta_{\kappa})\neq\varepsilon$. Then
\begin{align}\label{fdelta}
f_{\delta_{i}}\circ\varphi_{i}(x)&=\left\{
                                    \begin{array}{ll}
                                     f_{\delta_{i}}(x_i\chi_{\rbb^\kappa_{\varepsilon}}(x)), & \hbox{ if }i\neq j_0, \\
                                     f_+(-(x_{j_0}^2+1)\chi_{\rbb_{\varepsilon_{j_0}}}(x_{j_0})), & \hbox{ if }i=j_0.
                                    \end{array}
                                  \right.
\\\notag &=0.
\end{align}
By definition \eqref{defceps}, Lemma \ref{fmphi} (iii), and equality \eqref{fdelta}, we obtain
\begin{align*}
(\A_{\delta})^{\alpha}
=((M_{f_{\delta_j}\circ\varphi_j})_{j=1}^{\kappa})^{\alpha}=M_{\prod_{j=1}^{\kappa}(f_{\delta_j}\circ\varphi_j)^{\alpha_j}}\overset{\alpha_{i}\neq0}{=}0.
\end{align*}

(iii) Similarly to (ii), we show that
\begin{align*} (\A_\varepsilon)^\alpha=M_{\prod_{j=1}^{\kappa}(f_{\varepsilon_j}\circ\varphi_j)^{\alpha_j}}\notin\ogr{\hh},
\end{align*}
since the function $\prod_{j=1}^{\kappa}(f_{\varepsilon_j}\circ\varphi_j)^{\alpha_j}(x)=(-1)^{\alpha_{j_0}}x^{\alpha-\alpha_{j_0}e_{j_0}}(x_{j_0}^2+1)^{\alpha_{j_0}}\chi_{\rbb^\kappa_\varepsilon}(x)$ is not essentially bounded with respect to Lebesgue measure.

(iv)  Using Lemma \ref{fmphi} (iii), we get that
\begin{align*}
\A^{\alpha}=M_{\prod_{j=1}^{\kappa}(\varphi_j)^{\alpha_j}} \textmd{ and } \B^{\alpha}=M_{\prod_{j=1}^{\kappa}(\psi_j)^{\alpha_j}}.
\end{align*}
Define a function $h\colon\rbb^{\kappa}\to\rbb$ by the following formula
\begin{equation*}
h(x)=(1+|x_{j_0}|)^{-(\alpha_{j_0}+1)}\chi_{\Delta}(x),\quad
 x=(x_1,\ldots,x_{\kappa})\in\rbb^{\kappa},
\end{equation*}
where $\Delta=[-1,1]\times\ldots\times[-1,1]\times\underbrace{(-\infty,0]}_{j_0}\times[-1,1]\times\ldots\times[-1,1]$.
 Note that $h\in\dz{\B^{\alpha}}\backslash\dz{\A^{\alpha}}$. Indeed, by the definition of $h$ and Fubini theorem,
 we get
\begin{align*}
\int_{\rbb^{\kappa}}\Big\vert \Big(\prod_{j=1}^{\kappa}\psi_j^{\alpha_j}\Big)h\Big\vert^2\, \mathrm{dm}_\kappa(x)&=\int_{\rbb_\varepsilon^\kappa\cap\Delta} x^{2\alpha}\cdot\vert h(x) \vert^2\, \mathrm{dm}_\kappa(x)\\
&\leqslant\int_{(-\infty,0]} x_{j_0}^{2\alpha_{j_0}}(1+\vert x_{j_0}\vert)^{-2(\alpha_{j_0}+1)}\, \mathrm{dm}(x_{j_0})<\infty.
\end{align*}
Thus $h\in\dz{M_{\Pi_{j=1}^{\kappa}(\psi_j)^{\alpha_j}}}$. At the same time
\begin{align*}
\int_{\rbb^{\kappa}}\Big\vert \Big(\prod_{j=1}^{\kappa}\varphi_j^{\alpha_j}\Big)h\Big\vert^2\, \mathrm{dm}_\kappa(x)&=\int_{\rbb_\varepsilon^\kappa\cap\Delta} x^{2\alpha-2\alpha_{j_0}e_{j_0}}(x_{j_0}^2+1)^{2\alpha_{j_0}}\cdot\vert h(x) \vert^2\, \mathrm{dm}_\kappa(x)\\
&=\frac{(2\alpha_{j_0}+1)}{\prod_{j=1}^\kappa (2\alpha_j+1)}\int_{(-\infty,0]}\frac{(x_{j_0}^2+1)^{2\alpha_{j_0}}}{(1+\vert x_{j_0}\vert)^{2(\alpha_{j_0}+1)}}\, \mathrm{dm}(x_{j_0})\\ &=\infty,
\end{align*}
since $\alpha_{j_0}\neq0$. This means that $h\notin\dz{M_{\Pi_{j=1}^{\kappa}(\varphi_j)^{\alpha_j}}}$.
\end{exa}

\section{Positive $\kappa$-tuples}\label{kolejny}
In this section we study the multidimensional spectral order in the case of positive $\kappa$-tuples of pairwise commuting selfadjoint operators. We begin by the following proposition, which generalizes \cite[Proposition 7.1.]{Paneta2012}.
\begin{pro}\label{1901}
  Let $\A,\B\in\shk$  and
let $\alpha\in[0,\infty)^{\kappa}$. Suppose that $\A$
and $\B$ are positive. If $\A\Lec \B$, then the
following conditions hold:
\begin{enumerate}[{\em(i)}]
\item $\A^\alpha\Lec \B^\alpha$,
 \item $\dz{\B^\alpha}\subseteq \dz{\A^\alpha}$,
\item $\Vert \A^\alpha h\Vert \Le \Vert \B^\alpha h\Vert$ for every $h\in\dz{\B^\alpha}$,
\item $\is{\A^\alpha h}{h}\Le \is{\B^\alpha h}{h}$ for every $h\in \dz{\B^\alpha}$,
\item $\A^\alpha\Le \B^\alpha$.
\end{enumerate}
Moreover $\dzn{\B}\subseteq\dzn{\A}$ and
$\bscr(\B)\subseteq\bscr(\A)$.
\end{pro}
\begin{proof} (i) Apply Theorem \ref{spekwiel} for function $\psi_\alpha$ defined by equality \eqref{psibeta}, which is increasing.

(ii), (iii), (iv) and (v) follows from (i) and
\cite[Proposition 7.1.]{Paneta2012} applied for $s=1$ and
positive operators $\A^\alpha$ and $\B^\alpha$.

 To prove the "moreover" part let us note that
inclusion $\dzn{\B}\subseteq\dzn{\A}$ follows from
(ii) and Lemma \ref{dnieskonczony}. If
$h\in\bscr_{a}(\B)$ for some
$a\in(0,\infty)^{\kappa}$, then using (iii) and
Proposition \ref{wektogr} (b) we can find real number
$c>0$ such that
\begin{equation*}
 \Vert \A^\alpha h\Vert\Le \Vert \B^\alpha h\Vert\Le
ca^{\alpha},\: \textmd{for every}\:
\alpha\in\zbb_+^{\kappa},
\end{equation*} Applying Proposition \ref{wektogr} once more we get that  $h\in\bscr_{a}(\A)$, which gives us the second inclusion.
\end{proof}
The following two theorems Theorem \ref{2404} and Theorem \ref{2404nowe}, which are counterparts of \cite[Theorem 7.4.]{Paneta2012}, characterize multidimensional spectral order for positive $\kappa$-tuples in terms of appropriate inequalities for $\mathcal{C}^\infty$-vectors.
 In Theorem \ref{2404nowe} we assume additionally that one of the compared $\kappa$-tuples consists of injective operators.

\begin{thm}
  \label{2404}
 Let $\A=(A_1,\ldots,A_\kappa),\B=(B_1,\ldots,B_\kappa)\in\shk$ be positive.
Then the following conditions are equivalent\/{\em :}
   \begin{enumerate}[{\em(i)}]
   \item $\A \Lec \B$,
   \item $\dzn{\B} \subseteq \dzn{\A}$ and $L_{\A/\B}(h) \Le 1$ for every $h\in\dzn{\B}$,
   \item $\bscr(\B) \subseteq \dzn{\A}$ and $L_{\A/\B}(h) \Le 1$ for every $h\in\bscr(\B)$,
   \item $\bscr(\B) \subseteq \dzn{\A}$ and $L_{A_j/B_j}(h) \Le 1$ for every $h\in\bscr(\B)$ and $j=1,\ldots,\kappa$,
\item $\bscr(\B) \subseteq \bscr(\A)$ and $L_{\A/\B}(h) \Le 1$ for every $h\in\bscr(\B)$,
   \end{enumerate}
   where\footnote{We use the convention that $0/0=0$ and $a/0=\infty$ if $a>0$.}
\begin{equation*}L_{\A/\B}(h):=\displaystyle\limsup_{\vert \alpha\vert\to\infty}
\sqrt[\leftroot{1}\uproot{1} \vert \alpha\vert]{\Vert
\A^\alpha h \Vert / \Vert
\B^\alpha h \Vert}, \quad h \in
\dzn{\A} \cap\dzn{\B}.
\end{equation*}
\end{thm}
\begin{proof} (i)$\Rightarrow$(ii) and (i)$\Rightarrow$(v) follow from
Proposition \ref{1901}.

Implications (ii)$\Rightarrow$(iii), (iii)$\Rightarrow$(iv),  and (v)$\Rightarrow$(iii) are obvious.

 (iv)$\Rightarrow$(i) According to Theorem \ref{spekwiel}, it suffices to show that
   \begin{equation}\label{inkluzja1}
   \ob{F_{B_j}(x)}\subseteq\ob{F_{A_j}(x)} \text{ for every }x\in[0,\infty) \text{ and }j=1,\ldots,\kappa,
   \end{equation} since $\A,\B$ are positive. Fix $j\in\{1,\ldots,\kappa\}$ and $x\geqslant 0$.
  Let $h\in\ob{F_{B_j}(x)}$. Define $h_k:=E_{\B}(\Delta_k)h$ for $k\in\nbb$, where \begin{align*}
  \Delta_k=[0,k]\times\ldots\times[0,k]\times\underbrace{[0,x]}\limits_j\times[0,k]\times\ldots\times[0,k]\subset\rbb^\kappa.
   \end{align*} By Proposition \ref{wektogr}, $h_k\in\bscr(\B)$ for every $k\in\nbb$. In particular, by the assumption, Lemma \ref{dnieskonczony}, and \cite[Proposition A.1.]{Paneta2012}, $h_k\in \bscr(B_j)\cap\dzn{A_j}$ and $\lim_{n\to\infty}\sqrt[n]{\|A_j^nh_k\|/\|B_j^nh_k\|}= L_{A_j/B_j}(h_k)\Le 1$ for every $k\in\nbb$. This implies that there is $N\in\nbb$ such that $\Vert A_j^nh_k\Vert=\Vert B_j^nh_k\Vert=0$ or $\Vert B_j^nh_k \Vert\neq0$ for $n>N$.
 Then, by \cite[Lemma 8.]{StoSza1997}, we have
  \begin{align*} \lim_{n\to\infty}\sqrt[n]{\|A_j^nh_k\|}&= \lim_{n\to\infty}\sqrt[n]{\|A_j^nh_k\|/\|B_j^nh_k\|}\cdot\lim_{n\to\infty}\sqrt[n]{\|B_j^nh_k\|}\\&\Le \lim_{n\to\infty}\sqrt[n]{\|B_j^nh_k\|}\Le x.
  \end{align*}
  Hence, by \cite[Lemma 8.]{StoSza1997} and \cite[Proposition 5.1.]{Paneta2012}, $h_k\in \bscr_x(A_j)=\ob{F_{A_j}(x)}$ for every $k\in\nbb$. Note that $h_k\to h$ as $k\to\infty$, since $h\in\ob{F_{B_j}(x)}$ and $\bigcup_{k\in\nbb}\Delta_k=[0,\infty)\times\ldots\times\underbrace{[0,x]}\limits_j\times\ldots\times[0,\infty)$. Thus $h\in\ob{F_{A_j}(x)}$ as  $\ob{F_{A_j}(x)}$ is closed, which proves \eqref{inkluzja1}.
\end{proof}

\begin{thm}\label{2404nowe}
 Let $\A=(A_1,\ldots,A_\kappa),\B\in\shk$ be positive.
 Assume that $\jd{A_j}=\{0\}$ for $j=1,\ldots,\kappa$.
 If $\Lambda\subseteq[0,\infty)^\kappa$ satisfies the
 condition \eqref{lambdasup}, then the following
 conditions are equivalent\/{\em :}
   \begin{enumerate}[{\em(i)}]
   \item $\A \Lec \B$,
\item $\dzn{\B} \subseteq \dzn{\A}$ and $L_{\A/\B}(h) \Le 1$ for every $h\in\dzn{\B}$,
   \item $\bscr(\B) \subseteq \dscr^{\Lambda}(\A)$ and $L_{\A/\B}^{\Lambda}(h) \Le 1$ for every $h\in\bscr(\B)$,
\item $\bscr(\B) \subseteq \bscr(\A)$ and $L_{\A/\B}^{\Lambda}(h) \Le 1$ for every $h\in\bscr(\B)$,
   \end{enumerate}
where
\begin{equation*}L_{\A/\B}^{\Lambda}(h):=\displaystyle\limsup_{\substack {\vert \alpha\vert\to\infty\\ \alpha\in\Lambda}}
\sqrt[\leftroot{1}\uproot{1} \vert \alpha\vert]{\Vert
\A^\alpha h \Vert / \Vert
\B^\alpha h \Vert},\; h \in
\dscr^{\Lambda}(\A)
\cap\dscr^{\Lambda}(\B).
\end{equation*}
\end{thm}
\begin{proof}
(i)$\Rightarrow$(ii) This follows from
Proposition \ref{1901} and equality \eqref{ajn1}.

(ii)$\Rightarrow$(iii) It is obvious.

(iii)$\Rightarrow$(i) It suffices to
verify, that $F_{\B}(a)\Le F_{\A}(a)$ for every
$a\in[0,\infty)^\kappa$, since $\A$ and $\B$ are
positive. Fix $a\in [0,\infty)^{\kappa}$ and
$h\in\ob{F_{\B}(a)}$. Applying Proposition
\ref{wektogr} we get that
\begin{equation*}
h\in\bscr(\B) \textmd{ and }\Vert
\B^\alpha h\Vert\Le ca^\alpha, \quad \alpha\in\Lambda,
\end{equation*} for some real number $c>0$. By the assumptions
\begin{equation}\label{halfa4}
 h\in\dscr^\Lambda(\A)=\dscr^\Lambda(\A)\cap \Big(\bigcup_{j=1}^{\kappa}\jd{A_j}\Big)^\perp.
\end{equation}
 Fix $\varepsilon >0$. Since $L_{\A/\B}^{\Lambda}(h) \Le 1$, there exists a positive number $M\in\nbb$ such that
\begin{align*}
\sqrt[\leftroot{1}\uproot{1} \vert
\alpha\vert]{\Vert \A^\alpha h \Vert
/ \Vert \B^\alpha h \Vert}<1+\varepsilon \textmd{ for
every }
\alpha\in\Lambda_M:=\{\beta\in\Lambda\colon
\vert \beta\vert>M\}.
\end{align*} In particular,
   \begin{align*}
 \|\A^\alpha h\| \Le
 (1+\varepsilon)^{\vert \alpha\vert} \Vert
 \B^\alpha h \Vert
  \Le c(1+\varepsilon)^{\vert
\alpha\vert}a^{\alpha}=c[(1+\varepsilon)
a]^{\alpha}
   \end{align*}
for every $\alpha\in\Lambda_M$. Note that the set $\Lambda_M$ satifies the condition \eqref{lambdasup}. Hence, by
 \eqref{halfa4} and Proposition
\ref{wektogr}, we obtain that $h\in\ob
{F_\A((1+\varepsilon)a)}$. Since $\varepsilon>0$ was chosen arbitrarily, we infer that $h\in\ob
{F_\A(a)}$.

(i)$\Rightarrow$(iv) It follows from Proposition \ref{1901}.

 (iv)$\Rightarrow$(iii) Note that
$\bscr(\A)\subseteq\dscr^\Lambda(\A)$ which is the consequence of equality \eqref{ajn1}.
\end{proof}
The next proposition describes the multidimensional spectral order for positive elements in $\shk$ in terms of the classical order ''$\Le$''.
\begin{pro}\label{lambdanierowne2}
 Let $\A,\B\in\shk$ be positive.
  Suppose that there exists a family
$\{r_\alpha\}_{\alpha\in[0,\infty)^\kappa}\subseteq[1,\infty)$
such that
\begin{equation}\label{ralfa1}
\limsup_{\vert\alpha\vert\to\infty}\sqrt[\leftroot{1}\uproot{1} \vert
\alpha\vert]{r_\alpha}\Le 1.
\end{equation}
 Then the following conditions are equivalent\/{\em :}
\begin{enumerate}[{\em(i)}]
\item $\A\Lec \B$,
\item $\A^\alpha\Lec \B^\alpha$ for every $\alpha\in[0,\infty)^\kappa$,
\item $\A^\alpha\Le \B^\alpha$ for every $\alpha\in[0,\infty)^\kappa$,
\item $\A^\alpha\Le r_\alpha \B^\alpha$ for every $\alpha\in[0,\infty)^\kappa$.
\end{enumerate}
\end{pro}
\begin{proof} (i)$\Rightarrow$(ii) and (ii)$\Rightarrow$(iii) follow from Proposition \ref{1901}.

 (iii)$\Rightarrow$(iv) It is obvious.

 (iv)$\Rightarrow$(i) Let $h\in\bscr(\B)$. Then, by assumption and equality \eqref{ajn1},
$h\in\dz{\B^\frac{\alpha}{2}}\subseteq\dz{\A^\frac{\alpha}{2}}$
and
\begin{equation}\label{halfa3}
\Vert \A^\frac{\alpha}{2}h\Vert^2=\Vert
(\A^\alpha)^{\frac{1}{2}}h\Vert^2\Le r_\alpha\Vert
(\B^\alpha)^{\frac{1}{2}}h\Vert^2=r_\alpha\Vert
\B^\frac{\alpha}{2}h\Vert^2, \alpha\in[0,\infty)^\kappa.
\end{equation} In particular, $h\in\dzn{\A}$. Applying \eqref{halfa3} and \eqref{ralfa1} we obtain
\begin{align*} L_{\A/\B}(h)\Le\limsup_{\vert \alpha\vert\to\infty}\sqrt[\leftroot{1}\uproot{1} \vert \alpha\vert]{r_\alpha}\Le 1.
\end{align*}
Hence $\A\Lec \B$ by Theorem \ref{2404}.
\end{proof}
Repeating the proof of Proposition \ref{lambdanierowne2} and applying Theorem \ref{2404nowe} instead of Theorem \ref{2404}, we have the following proposition.
\begin{pro}\label{lambdanierowne1}
 Let $\A,\B\in\shk$ be positive and let $\Lambda\subseteq[0,\infty)^\kappa$ such that $\jd{A_j}=\{0\}$ for
$j=1,\ldots,\kappa$,
Suppose that the set
$\Lambda\subseteq[0,\infty)^\kappa$ satisfies condition \eqref{lambdasup},
and there exists a family
$\{r_\alpha\}_{\alpha\in\Lambda}\subseteq[1,\infty)$
satisfying the following condition
\begin{equation*}
\limsup_{\substack {\vert \alpha\vert\to\infty\\
\alpha\in\Lambda}}\sqrt[\leftroot{1}\uproot{1} \vert
\alpha\vert]{r_\alpha}\Le 1.
\end{equation*}
 Then $\A\Lec \B$ if and only if $\A^\alpha\Le r_\alpha \B^\alpha$ for every $\alpha\in\Lambda$.
\end{pro}

 Let $\A, \B\in\shk$ be positive. Let us consider the set
\begin{equation*}\Lambda(\A,\B):=\{\alpha\in[0,\infty)^{\kappa}\colon \A^\alpha\Le \B^\alpha\}.
\end{equation*} Proposition \ref{1901} tells us, that $\Lambda(\A,\B)=[0,\infty)^\kappa$ if $\A\Lec\B$. Taking into account \cite[Corollary 7.6.]{Paneta2012}, we are going to determine what are sufficient and necessary condition on the set $\Lambda(\A,\B)$ for the relation $\A\Lec\B$ to hold.
\begin{pro}
 Let $\A=(A_1,\ldots,A_{\kappa}), \B=(B_1,\ldots,B_{\kappa})\in\shk$ be positive. Cosider the following conditions:
\begin{enumerate}[{\em (i)}]
\item $\A\Lec\B$,
\item   the set $\Lambda(\A,\B)\cap\{se_j\colon s\in[0,\infty)\}$
is unbounded for every $j=1,\ldots,\kappa$.
\item  \begin{equation}\label{lambdasup1}
 \sup_{\alpha\in\Lambda(\A,\B)}\frac{\alpha_j}{1+|\alpha|}=1, \quad j=1,\ldots,\kappa.
 \end{equation}
\end{enumerate}
Then conditions {\em(i)} and {\em(ii)} are equivalent.
Moreover, if $\jd{A_j}=\{0\}$ for $j=1,\ldots,\kappa$, then all the conditions {\em(i)-(iii)} are equivalent.
\end{pro}
\begin{proof}
(i)$\Rightarrow$(ii) It is a consequence of the
Proposition \ref{1901}.

(ii)$\Rightarrow$(i) Fix $j\in\{1,\ldots,\kappa\}$. By
\cite[Lemma 6.5.2]{Bir-Sol} we get that
$C_j^s=\mathbf{C}^{se_j}$ for every positive $\mathbf{C}\in\shk$
and for every $s\in[0,\infty)$, since
$\psi_{se_j}(x)=\psi_s\circ\pi_j(x)$ for every
$x\in[0,\infty)^\kappa$. Thus the set $\{s\in[0,\infty)\colon A_j^s\Le B_j^s\}$ is infinite for every $j=1,\ldots,\kappa$. From \cite[Corollary 7.6]{Paneta2012} we infer that
 $A_j\Lec B_j$ for every $j=1,\ldots,\kappa$.
Hence $\A\Lec\B$ by Theorem \ref{spekwiel}.

Now assume that $\jd{A_j}=\{0\}$ for $j=1,\ldots,\kappa$. The equivalence (i)$\Leftrightarrow$(iii) follows from Proposition
\ref{lambdanierowne1} applied for
$\Lambda=\Lambda(\A,\B)$ and
$\{r_\alpha\}_{\alpha\in\Lambda}$, where $r_\alpha=1$
for every $\alpha\in\Lambda$.
\end{proof}
Note that in the case $\kappa=1$ conditions (ii) and (iii) coincide.
In particular, injectivity of $A_1$ can be omitted by \cite[Corollary 7.6]{Paneta2012}. On the other hand,  the condition \eqref{lambdasup1} may be to weak to ensure inequality $\A\Lec\B$ if $\kappa>1$. This can happen if some of the operators $A_1,\ldots,A_\kappa$ are not injective.
\begin{exa} Let $\hh=\cbb^2$ be a Hilbert space with standard orthonormal basis $\{(1,0),(0,1)\}$ and let $\theta\in[1,\infty)$. Consider the $2\times2$ matrices $A_1$, $A_2$, $B_{\theta,1}$, $B_{\theta,2}$ given by
\begin{equation*}
A_1=\left[
    \begin{array}{cc}
      0 & 0 \\
      0 & 0 \\
    \end{array}
  \right], \quad
  A_2=\left[
    \begin{array}{cc}
      2 & 1 \\
      1 & 2 \\
    \end{array}
  \right], \quad
    B_{\theta,1}=\left[
    \begin{array}{cc}
      1 & 0 \\
      0 & 1 \\
    \end{array}
  \right], \text{ and }
   B_{\theta,2}=\left[
    \begin{array}{cc}
      3 & 1 \\
      1 & 1+\theta \\
    \end{array}
  \right].
\end{equation*}
Let us define $\A:=(A_1,A_2)$ and $\B_\theta:=(B_{\theta,1},B_{\theta,2})$. It is evident that $\A, \B_\theta\in\B_{cs}(\hh,2)$ are positive for $\theta\in[1,\infty)$, $\jd{A_1}=\hh$, and $\jd{A_2}=\{0\}$.  We are going to show, that $\A$ and $\B_\theta$ satisfy the following conditions:
\begin{enumerate}[(i)]
\item for every $k\in\nbb$ there exixts $\theta_k\in(2,\infty)$ such that
\begin{equation*}
\Big((0,\infty)\times[0,\infty)\Big)\cup\Big(\{0\}\times[0,k]\Big)\subseteq\Lambda(\A,\B_\theta) \text{ for all } \theta \in [\theta_k, \infty),
\end{equation*}
\item $\A\Lec \B_\theta$ if and only if $\theta=2$.
\end{enumerate}
In particular, if $\theta\neq 2$, then $\A\not\Lec \B_\theta$
although $\A$ and $\B_\theta$ satisfy
condition \eqref{lambdasup1}.\end{exa}
\begin{proof}
(i) First, if
$(s,t)\in(0,\infty)\times[0,\infty)$, then
$A_1^sA_2^t=0$ and
$B_{\theta,1}^sB_{\theta,2}^t=B_{\theta,2}^t\Ge 0$. In
particular
$(s,t)\in\Lambda(\A,\B_\theta)$.
On the other hand, by \cite[Proposition 8.3.]{Paneta2012}
there exists $\theta_k\in(2,\infty)$ such that
 \begin{equation*} \left[
                     \begin{array}{cc}
                       1 & 1 \\
                       1 & 1 \\
                     \end{array}
                   \right]^n\Le \left[
                                  \begin{array}{cc}
                                    2 & 1 \\
                                    1 & \theta \\
                                  \end{array}
                                \right]^n,\quad n=0,\ldots,k, \text{ and }\theta\in[\theta_k,\infty).\end{equation*}
  This implies that
  \begin{align*}\left[
    \begin{array}{cc}
      2 & 1 \\
      1 & 2 \\
    \end{array}
  \right]^k&=\left(\left[
    \begin{array}{cc}
      1 & 1 \\
      1 & 1 \\
    \end{array}
  \right]+\left[
    \begin{array}{cc}
      1 & 0 \\
      0 & 1 \\
    \end{array}
  \right]\right)^k=\sum_{j=0}^k\left(\begin{array}{c}k\\ j
\end{array}\right)\left[
    \begin{array}{cc}
      1 & 1 \\
      1 & 1 \\
    \end{array}
  \right]^j\\
  &\Le \sum_{j=0}^k\left(\begin{array}{c}k\\ j
\end{array}\right)\left[
                                  \begin{array}{cc}
                                    2 & 1 \\
                                    1 & \theta \\
                                  \end{array}
                                \right]^j= \left[
    \begin{array}{cc}
      3 & 1 \\
      1 & 1+\theta \\
    \end{array}
  \right]^k.\end{align*}
 By L\"{o}wner-Heinz inequality we get that
$A_2^s=(A_2^k)^{\frac{s}{k}}\Le
(B_{\theta,2}^k)^{\frac{s}{k}}=B_{\theta,2}^s$ for
every $s\in[0,k]$. Thus
  \begin{equation*}
  \{0\}\times[0,k]\subseteq\Lambda(\A,\B_\theta).
  \end{equation*}

 (ii) We deduce from Theorem \ref{spekwiel} and Corollary
\ref{summno} that the inequality $\A\Lec \B_\theta$
holds if and only if $A=A_2-I_\hh\Lec
B_{\theta,2}-I_\hh=B_{\theta}$, since $A_1\Lec
B_{\theta,1}$. The application of \cite[Lemma 8.2.]{Paneta2012}
completes the proof.\end{proof}

   \appendix
\section{Resolutions of the identity on $\rbb^\kappa$}\label{app1}
In this section, we recall the definition of
an abstract multivariable resolution of the identity and we
outline an idea how to construct
spectral measures on $\rbb^\kappa$ by using multivariable resolutions of the identity (see \cite{ash} for the case of scalar
measures).

For $c,d\in\rbb$ and $j\in\{1,\ldots,\kappa\}$, we
define the difference operator
$\bigtriangleup_{d,c}^{(j)}$ acting on operator-valued
functions $F\colon\rbb^\kappa\to\ogr{\hh}$ by
   \begin{align*} (\bigtriangleup_{d,c}^{(j)}
F)(x_1, \ldots, x_\kappa) & = F(x_1,\ldots,x_{j-1},d,
x_{j+1},\ldots,x_\kappa)
   \\
& \hspace{5ex} - F(x_1,\ldots,x_{j-1}, c,
x_{j+1},\ldots,x_\kappa)
   \end{align*}
for all $x_1,\ldots,x_\kappa\in\rbb$. If
$F\colon\rbb^\kappa\to\ogr{\hh}$, then we set
   \begin{equation*}
F(a,b]= \bigtriangleup_{b_1,a_1}^{(1)}\cdots
\bigtriangleup_{b_\kappa,a_\kappa}^{(\kappa)}F, \quad
a, b \in \rbb^\kappa.
   \end{equation*}
We say that an operator-valued function
$F\colon\rbb^\kappa\to\rzut{\hh}$ is a {\em resolution of the identity} on $\rbb^\kappa$ if $F$ satisfies the
following three conditions\footnote{Given a sequence
$\{a_n\}_{n=1}^\infty\subseteq\rbb^\kappa$ and
$a\in\overline{\rbb}^\kappa$, we write $a_n\searrow
a$ (resp., $a_n\nearrow a$) if
$\{a_n\}_{n=1}^\infty$ is monotonically decreasing
(resp., increasing) with respect to the partial order
``$\Le$'' in $\rbb^\kappa$, and convergent to $a$.
If $a = (-\infty, \ldots, -\infty)$ (resp., $a =
(\infty, \ldots, \infty)$), we also write $a_n\searrow
- \infty$ (resp., $a_n\nearrow \infty$).}:
   \begin{enumerate}[(A)]
   \item $\big(F(a,b]\big)(x)\in\rzut{\hh}$ for all
$a,b,x\in\rbb^\kappa$ such that $a\Le b$,
   \item $\textsc{sot-}\lim\limits_{\substack{x\to x_0
\\ x_0\Le x}} F(x)=F(x_0)$ for every $x_0\in\rbb^\kappa$,
   \item if $\{a_n\}_{n=1}^{\infty}, \{b_n\}_{n=1}^{\infty}
\subseteq \rbb^\kappa$ are such that $a_n\searrow
-\infty$ and $b_n\nearrow \infty$, then
   \begin{align*}
\textsc{sot-}\lim\limits_{n\to\infty}F(a_n,b_n]=I.
   \end{align*}
   \end{enumerate}

If $E$ is a spectral measure on
$\borel(\rbb^\kappa)$, then the function
$F\colon\rbb^\kappa\to\rzut{\hh}$ given by
\begin{equation}\label{fab}
F(x)=E((-\infty,x]),\quad x\in\rbb^\kappa,
\end{equation} is a resolution of the identity.
 In fact, the equation \eqref{fab} defines a one-to-one correspondence between the set of all resolutions of the identity on
$\rbb^\kappa$ and the set of all spectral measures on $\borel(\rbb^\kappa)$, which can be deduced from the following theorem.
\begin{thm} Let $F$ be a resolution of the identity on $\rbb^\kappa$. Then there exists the unique spectral measure on
$\borel(\rbb^\kappa)$ such that
\begin{equation*}
E((a,b])=F(a,b],\quad a,b\in\rbb^\kappa,\: a\Le b.
\end{equation*}
\end{thm}
\begin{proof}[The sketch of the proof]
{\it The uniqueness.} The uniqueness is a consequence of the fact that the family $\{(a,b]\colon a,b\in\rbb^\kappa,\: a\Le b\}$ generates the $\sigma$-algebra $\borel(\rbb^\kappa)$.

{\it The existence.} Denote by
$\mathcal{F}_0(\rbb^\kappa)$ the collection of all
finite disjoint sums of sets of the form $(a, b]\cap \rbb^\kappa$, where $a,b \in
\overline{\rbb}^{\kappa}$. Clearly,
$\mathcal{F}_0(\rbb^\kappa)$ is the algebra generated
by the family $\mathcal{I}=\{(a, b] \cap \rbb^\kappa \colon a,b \in
\overline{\rbb}^{\kappa}\}$.
Now we define the function $E^0\colon \mathcal{F}_0(\rbb^\kappa)\to\rzut{\hh}$ in few  steps:\\
{\it Step} 1. If $a,b\in\rbb^\kappa$ and $a\Le b$, then $E^0((a,b]):=F(a,b]$.\\
{\it Step} 2. If $a,b\in\overline{\rbb}^\kappa$ and $a\Le b$, then
$E^0((a,b]):=\textsc{sot-}\lim F(a_n,b_n]$,
where $a_n,b_n\in\rbb^\kappa$ for $\nbb$,
$a_n\searrow a$ and $b_n\nearrow b$. The limit exists since $F(a,b]\leqslant F(c,d]$ whenever $c\leqslant a\leqslant b\leqslant d$.\\
{\it Step} 3. Now if $J=\bigcup_{j=1}^{n}J_j$, where $J_1,\ldots,J_n\in\mathcal{I}$ are pairwise disjoint, then we set
$E^0(J):=\sum_{j=1}^{n}E^0(J_j)$. $E^0$ is well-defined and additive. Moreover,
$E^0(\rbb^\kappa)=I$ by (C).

 By \cite[Theorem 5.2.3.]{Bir-Sol} it is sufficient to show that $E^0$ is $\sigma$-additive. Fix $h\in\hh$. Let
$\mu_h^0(J):=\is{E^0(J)h}{h}$ for
$J\in\mathcal{F}_0(\rbb^\kappa)$ and
$F_h(x):=\is{F(x)h}{h}$ for $x\in\rbb^\kappa$.
By the conditions (A) and (B), $F_h$ is a distribution function
on $\rbb^\kappa$. Moreover,
$\mu_h^0((a,b])=F_h(a,b]$ for
$a,b\in\rbb^\kappa,\: a\Le b$. By
\cite[Theorem 1.4.9.]{ash} there exists a Borel measure $\mu_h$ such that $\mu_h((a,b])=F_h(a,b]$ for $a,b\in\rbb^\kappa,\: a\Le b$. In particular,
$\mu_h|_{\mathcal{F}_0(\rbb^\kappa)}=\mu_h^0$. Thus $\mu_h^0$ is $\sigma$-additive. Therefore
$E^0$ is also $\sigma$-additive.
\end{proof}

\section*{Acknowledgements}
I would like to thank Professor Jan Stochel for his helpful advice, which improved the paper.

\bibliographystyle{plain}
   \bibliography{library}

   \end{document}